\documentclass[11pt]{amsart}

\usepackage[mathcal]{eucal}
\usepackage{amssymb}
\usepackage{graphicx}
\usepackage[all]{xy}

\newtheorem{thm}{Theorem}
\newtheorem{prop}[thm]{Proposition}
\newtheorem{lem}[thm]{Lemma}

\theoremstyle{remark}
\newtheorem{rem}[thm]{Remark}

\theoremstyle{definition}
\newtheorem{defi}[thm]{Definition}

\newcommand{\C}{\mathbb C}
\newcommand{\R}{\mathbb R}
\newcommand{\Z}{\mathbb Z}

\newcommand{\HH}{\mathbb H}

\newcommand{\Id}{\mathrm{Id}}
\newcommand{\End}{\mathrm{End}}

\newcommand{\T}{\mathcal T}

\newcommand{\K}{\mathcal K}
\newcommand{\M}{\mathcal M}
\newcommand{\h}{\mathcal H}

\newcommand{\la}{\lambda}
\newcommand{\g}{\gamma}

\newcommand{\s}{\sigma}

\renewcommand{\phi}{\varphi}

\title[Factorization Rules in Quantum Teichm\"uller Theory]
{Factorization Rules in Quantum Teichm\"uller Theory}
\author{Julien Roger}
\address{Department of Mathematics, Rutgers University, New Brunswick NJ~08854}
\email{juroger@math.rutgers.edu}
\thanks{This research was partially supported by the grant DMS-0604866 
from the National Science Foundation.}

\begin{document}

\begin{abstract}
For a punctured surface $S$, a point of its Teichm\"uller space $\mathcal{T}(S)$ determines an irreducible representation of its quantization $\mathcal{T}^q(S)$. We analyze the behavior of these representations as one goes to infinity in $\mathcal{T}(S)$, or in the moduli space $\mathcal{M}(S)$ of the surface. The main result of this paper states that an irreducible representation of $\mathcal{T}^q(S)$ limits to a direct sum of representations of $\mathcal{T}^q(S_\gamma)$, where $S_\gamma$ is obtained from $S$ by pinching a multicurve $\gamma$ to a set of nodes. The result is analogous to the factorization rule found in conformal field theory.
\end{abstract}

\maketitle

Let $S$ be an oriented surface of genus $g$ obtained from a closed compact surface $\overline{S}$ by removing $s$ punctures $v_1$, \ldots, $v_s$. The \emph{Teichm\"uller space} $\T(S)$ of $S$ is the space of isotopy classes of complete hyperbolic metrics on $S$ with finite area. It comes equipped with a natural K\"ahler metric, called the Weil--Petersson metric which is invariant under the action of the mapping class group $MCG(S)$ onto $\T(S)$. A quantization of the Teichm\"uller space was successfully described by L. Chekhov and V. V. Fock in \cite{CheFo}, and, in a slightly different setting, by R. Kashaev in \cite{Ka1}. In the work of Chekhov and Fock, the main geometric ingredient is the notion of shear coordinates on the enhanced Teichm\"uller space $\widetilde{\T}(S)$ which were introduced by W. Thurston \cite{Thurston}. On the algebraic side, they make use of the quantum dilogarithm as described by L. Faddeev and Kashaev \cite{FaKa}.

In the physics literature, the interest for the quantization of Teichm\"uller theory can be traced back  to the work of E. Verlinde and H. Verlinde \cite{Ver1,Ver2} among others. In particular, H. Verlinde conjectured in \cite{Ver2} that quantum Teichm\"uller theory should give rise to a family of representations of the mapping class groups which could be identified with a modular functor obtained from Liouville conformal field theory. The existence of such a modular functor associated to the quantum Teichm\"uller space was conjectured also by Fock \cite{Fo} and was studied further by J. Teschner \cite{Tesch1}. A generalized version of this conjecture was made by Fock and A. B. Goncharov \cite{FoGon} in their study of (quantum) higher Teichm\"uller theory.

The goal of this paper is to investigate a similar question in the context of the exponential version of the quantum Teichm\"uller space studied by H. Bai, F. Bonahon and X. Liu \cite{Liu1,Bai,BoLiu}. Given a parameter $q\in\C^*$, The quantum Teichm\"uller space $\T^q(S)$ is a non-commutative algebra, deformation of the algebra of functions on $\T(S)$. For $q$ a root of unity, Bonahon and Liu \cite{BoLiu} describe a complete classification of the finite dimensional irreducible representations $\rho\colon\T^q(S)\to\End(V)$. In particular, in the case when $q$ is a primitive $N$--th root of unity with $N$ odd and fixing weights $p_1,\ldots,p_s\in\{0,\ldots,N-1\}$ labelling the punctures $v_1$, \ldots, $v_s$ of $S$, one can associate to every hyperbolic metric $m\in\T(S)$ a unique irreducible representation $\rho_m\colon\T^q(S)\to\End(V)$. Using a construction similar to the one described in \cite{BaBoLiu}, one can then construct a projective vector bundle $\mathcal{K}^q=\mathcal{\K}^q(p_1,\ldots,p_s)$ over $\T(S)$ with fiber $\mathbb{P}V$, where the fiber at $m\in\T(S)$ is endowed with the action of the irreducible representation $\rho_m$ of $\T^q(S)$. This construction behaves well under the action of the mapping class group $MCG(S)$ and we obtain a projective vector bundle $\widetilde{\mathcal{K}}^q$ over the moduli space $\mathcal{M}(S)=\T(S)/MCG(S)$. 

In the spirit of conformal field theory, one can then ask if this bundle extends to  the Deligne--Mumford compactification $\overline{\mathcal{M}(S)}$ of the moduli space. To study this question, we analyze how the representation $\rho_m$ of $\T^q(S)$ breaks down when the metric $m$ approaches a point of $\overline{\mathcal{M}(S)}\smallsetminus\mathcal{M}(S)$, that is, when the lengths of a finite number of geodesics of $S$ tend to 0 for this metric. The result we obtain can be interpreted as a \emph{factorization rule} for this theory.

More precisely, let $\la$ be an \emph{ideal triangulation} of $S$, that is, a triangulation of $S$ with vertices at the punctures. Following the construction of Chekhov and Fock in \cite{CheFo}, Bonahon and Liu \cite{Liu1,BoLiu} associate to $S$, the triangulation $\la$ and a parameter $q\in\C^*$, an algebra $\T^q_\la(S)$ called the \emph{Chekhov--Fock algebra}. It is the skew-commutative algebra over $\C$ with generators $X^{\pm1}_1,\ldots,X^{\pm1}_n$ associated to the edges of $\la$ and relations $X_i X_j=q^{2\sigma_{ij}}X_j X_i$, where the $\sigma_{ij}\in\left\{-2,-1,0,1,2\right\}$ are the coefficients of the Weil--Petersson Poisson structure on the enhanced Teichm\"uller space $\widetilde{\T}(S)$, parametrized using Thurston's shear coordinates.

Approaching a point in the boundary of $\overline{\mathcal{M}(S)}$ corresponds to shrinking a finite number of non-intersecting geodesics of $S$ to points. Hence we give ourselves a finite union of non-intersecting, non homotopic, essential simple closed curves $\g=\cup_i\g_i\subset S$ and we consider the surface $S_\g=S\smallsetminus\g$. One should think of $S_\g$ as being obtained from $S$ by pinching the multicurve $\g$ to a set of nodes and removing them. It is a possibly disconnected surface with two new punctures for each curve removed. The ideal triangulation $\la$ on $S$ induces an ideal triangulation $\la_\g$ on $S_\g$ whose edges are given by taking homotopy classes of edges of $\la\cap S_\g$.

An essential step is to relate the quantum Teichm\"uller spaces of $S$ and $S_\g$.
\begin{prop}
For every ideal triangulation $\la$ and every multicurve $\g$ on $S$, there exists an algebra homomorphism
\[\Theta^q_{\g,\la}:\T^q_{\la_\g}(S_\g)\longrightarrow\T^q_\la(S)\]
described explicitly by sending each generator of $\T^q_{\la_\g}(S_\g)$ to certain monomials in $\T^q_\la(S)$.
\end{prop}

The existence of this homomorphism is the algebraic translation of the fact that the Weil--Petersson Poisson structure extends naturally to the completion $\overline{\T(S)}$ of the Teichm\"uller space for the Weil--Petersson metric (see H. Masur \cite{Mas1} and S. Wolpert \cite{Wol}). This completion is called the \emph{augmented Teichm\"uller space} and was introduced by W. Abikoff \cite{Abi1} and L. Bers \cite{Ber1}. As a set, it is the union of $\T(S)$ and of the $\T(S_\g)$ for every multicurve $\g$. The action of $MCG(S)$ on $\T(S)$ extends to $\overline{\T(S)}$ and the quotient $\overline{\T(S)}/MCG(S)$ can be identified topologically with $\overline{\mathcal{M}(S)}$. The key geometric ingredient is given by an extension of the shear parameters on $\T(S)$ to the strata $\T(S_\g)$ of this augmentation. 

When $q$ is a primitive $N$--th root of unity with $N$ odd, Bonahon and Liu associate an irreducible representation $\rho_m:\T^q_\la(S)\rightarrow\End(V)$ of the Chekhov--Fock algebra to each metric $m\in\T(S)$ and weights $p_1,\ldots,p_s\in\{0,\ldots,N-1\}$ labelling the punctures $v_1$, \ldots, $v_s$ of $S$. $\rho_m$ is defined uniquely up to isomorphism and varies continuously with  the metric $m$. We proceed with studying the behavior of $\rho_m$ when $m$ approaches a lower-dimensional stratum $\T(S_\g)$ in $\overline{\T(S)}$.

For simplicity, let us restrict attention in the introduction to the case when $\g$ consists of a single curve.

\begin{thm}
\label{thmintro1}
Let $m_t\in\T(S)$ be a continuous family of hyperbolic metrics such that, as $t\to 0$, $m_t$ converges to $m_\g\in\T(S_\g)$ in $\overline{\T(S)}$. Let $\rho_t\colon\T^{q}_{\la}(S)\to\End(V)$ be a continuous family of irreducible representations classified by $m_t$ and weights $p_1,\ldots,p_s\in\{0,\ldots,N-1\}$ labelling the punctures $v_1$, \ldots, $v_s$ of $S$.

Then, as $t\rightarrow 0$, the representation \[\rho_t\circ\Theta^q_{\g,\la}:\T^{q}_{\la_\g}(S_\g)\rightarrow \End(V)\]
approaches \[\bigoplus^{N-1}_{i=0}\rho^i_\g:\T^{q}_{\la_\g}(S_\g)\rightarrow\End(\bigoplus^{N-1}_{i=0}V_i)\]
 where, for each $i$, $\rho^i_\g\colon\T^{q}_{\la_\g}(S_\g)\to\End(V_i)$ is the irreducible representation classified by $m_\g$, the  weights $p_1,\ldots,p_s$ labelling the old punctures and the weight $i$ labelling the two new punctures $v'$ and $v''$ of $S_\g$.
\end{thm}

The next step is to show that this decomposition is well-behaved under changes of triangulations. More precisely, for any two triangulations $\la$ and $\la'$ of $S$, Chekhov and Fock introduce \emph{quantum coordinate change isomorphisms} $\Phi^q_{\la\la'}\colon\widehat{\T}^q_{\la'}(S)\to\widehat{\T}^q_\la(S)$ between the fraction division algebras associated to the Chekhov--Fock algebras $\T^q_{\la}(S)$ and $\T^q_{\la'}(S)$. The \emph{quantum Teichm\"uller space} $\T^q(S)$ of $S$ is then defined to be the union of the $\widehat{\T}^q_\la(S)$ for every triangulation $\la$, modulo the relation which identifies $X'\in\widehat{\T}^q_{\la'}(S)$ to $\Phi^q_{\la\la'}(X')\in\widehat{\T}^q_\la(S)$.

The coordinate change isomorphisms are only defined between fraction algebras. For certain representations $\rho_\la$ of $\T^q_\la(S)$ the composition with $\Phi^q_{\la,\la'}$ still makes sense and we obtain a representation $\rho_\la\circ\Phi^q_{\la,\la'}\colon\T^q_{\la'}(S)\to\End(V)$. A representation $\rho$ of $\T^q(S)$ is then a family of representations $\{\rho_\la\}_\la$ of $\T^q_\la(S)$ such that $\rho_\la\circ\Phi^q_{\la\la'}=\rho_{\la'}$ for every triangulations $\la$ and $\la'$. In particular, given $m\in\T(S)$ and weights $p_1$,\ldots, $p_s$ labelling the punctures of $S$, one obtains a representation $\rho_m=\{\rho_{m,\la}\}_\la$ where $\rho_{m,\la}\colon\T^q_\la(S)\to\End(V)$ is the irreducible representation of $\T^q_\la(S)$ classified by these data.
  
\begin{thm}
\label{thmintro2}
Let $\rho_t=\{\rho_{t,\la}\}_\la$ be a continuous family of irreducible representations of $\T^q(S)$ classified by weights $p_1$, \ldots, $p_s$ and a continuous family $m_t\in\T(S)$ such that $m_t$ approaches $m_\g\in\T(S_\g)$ as $t\rightarrow 0$. For each triangulation $\la$ of $S$, we let the limit
\[
\lim_{t\to0}\rho_{t,\la}\circ\Theta^q_{\g,\la}=\bigoplus_i\rho^i_{\g,\la}
\]
be given as in Theorem~\ref{thmintro1}.

Then, for any triangulations $\la$ and $\la'$ and weight $i$, we have
\[\rho^i_{\g,\la'}=\rho^i_{\g,\la}\circ\Phi^q_{\la_\g\la'_\g}.\]

Hence the family of representations $\{\rho^i_{\g,\la}\}_\la$ determines an irreducible representation $\rho^i_\g\colon\T^q(S_\g)\to\End(V)$, which is classified by $m_\g$, the weights $p_1$, \ldots, $p_s$ labelling the old punctures, and the weight $i$ labelling the new punctures $v'$ and $v''$ of $S_\g$.
\end{thm}

\noindent
\textbf{Acknowledgments:} Most of the content of this paper was written as a graduate student at the University of Southern California under the supervision of Francis Bonahon. It is with great pleasure that I thank him for his support and continued interest in my work. His healthy skepticism during our frequent conversations made my modest victories all the more rewarding. I would also like to thank Ko Honda and Bob Penner for their support, as well as my fellow students at the time, Roman Golovko and Dmytro Chebotarov, for our frequent mathematical discussions. I would also like to thank St\'ephane Baseilhac and Feng Luo for several enlightening conversations regarding this work.



\section{Geometric background}

Throughout this paper, $S$ will be an oriented surface of genus $g$ obtained from a closed surface $\overline{S}$ without boundary by removing $s$ punctures $v_1,\ldots,v_s$. We will assume that $S$ has at least one puncture and has Euler characteristic $\chi(S)=2-2g-s<0$. The simplest such surfaces are the spheres with 3 or 4 punctures and the once-punctured torus.

\subsection{Teichm\"uller spaces}
For the purpose of this paper we will need two variants of Teichm\"uller space. The first and most classical one, denoted simply as $\T(S)$, will be the set of isotopy classes of complete hyperbolic metrics on $S$ with finite area. We will call this space simply the \emph{Teichm\"uller space} of $S$. However, we will need to drop the finite area condition to describe the exponential shear coordinates, which are essential to the definition of the quantum Teichm\"uller space. Let $Conv(S,m)$ denote the \emph{convex core} of $S$, that is, the smallest non-empty closed convex subset of $(S,m)$. $Conv(S,m)$ is a surface with cusps and geodesic boundaries and is homeomorphic to $S$. If $(S,m)$ has finite area then its convex core consists of the whole surface. Otherwise, some of the punctures  of $(S,m)$ will have a neighborhood isometric to an infinite area funnel bounded by one of the geodesic boundaries of $Conv(S,m)$. We let $\widetilde{\T}(S)$ be the space of isotopy classes of complete hyperbolic metrics on $S$, possibly with infinite volume, together with an orientation of each of the boundary components of $Conv(S,m)$. $\widetilde{\T}(S)$ is called the \emph{enhanced Teichm\"uller space} of $S$. In particular, since, for a complete hyperbolic metric, $Conv(S,m)$ has no boundary component, there is a natural embedding of $\T(S)$ into $\widetilde{\T}(S)$. 

\subsection{The augmented Teichm\"uller space}
We will also need to consider the \emph{augmented Teichm\"uller space} $\overline{\T(S)}$ which was introduced by Abikoff \cite{Abi1,Abi2} and Bers \cite{Ber1}, and was further studied by Masur \cite{Mas1} (See more recently Wolpert \cite{Wol2} and references therein). We will briefly recall its construction and some of its properties.

Let $\g=\g_1\cup\cdots\cup\g_k$ be the union of $k$ disjoint, non-homotopic, essential simple closed curves. Such a $\g$ will be called a \emph{multicurve}. Alternatively, $\g$ corresponds to a $(k-1)$-simplex in $C(S)$, the \emph{complex of curves} of $S$. We will denote by $S_\g$ the surface obtained from $S$ by removing the multicurve $\g$. It is a possibly disconnected surface with two new punctures for each curve removed. As a set, we define
\[\overline{\T(S)}=\T(S)\cup\bigcup_{\g\in C(S)}\T(S_{\g}),
\]
where $\T(S_\g)$ is the product of Teichm\"uller spaces associated to the connected components of $S_\g$. The $\T(S_\g)$ are called the \emph{strata} of $\overline{\T(S)}$.

A topology on $\overline{\T(S)}$ can be defined as follows: a sequence of metrics $(m_n)_n$ in $\T(S)$ converges to $m_\g\in\T(S_\g)$ if, as $n\rightarrow\infty$, the length $l_{m_n}(\g_i)$ of (the geodesic representative of) $\g_i$ for $m_n$ tends to 0 for every $i$, and $m_n$ converges uniformly to $m_\g$ on every compact subset of $S_\g$.

The action of the mapping class group $MCG(S)$ of $S$ on the Teichm\"uller space extends to $\overline{\T(S)}$ and the quotient $\overline{\T(S)}/MCG(S)$ can be identified, as a topological space, with the Deligne--Mumford compactification $\overline{\M(S)}$ of the moduli space $\M(S)=\T(S)/MCG(S)$ \cite{Abi1}.

\subsection{Exponential shear coordinates}

One of the main ingredients for the quantization of Teichm\"uller space as described first in \cite{CheFo} is the notion of shear coordinates introduced by W. Thurston \cite{Thurston} (see also \cite{Bo1}). We will describe here their exponential version, following for example \cite{Liu1}.

Since $\chi(S)<0$ and $S$ has at least one puncture, it admits an \emph{ideal triangulation} $\lambda=\left\{\la_1,\ldots,\la_n\right\}$, that is, a triangulation of $\bar{S}$ with vertices $v_1,\ldots,v_s$ and edges  $\la_1,\ldots,\la_n$, where the edges are considered up to isotopy. The number of edges of an ideal triangulation depends only on the Euler characteristic of $S$ and is given by $n=-3\chi(S)=6g+3s-6$. If we endow $S$ with a hyperbolic metric $m$, each edge $\la_i$ is isotopic to a unique geodesic $g_i$ for this metric. One can then associate to $\la_i$ a number $x_i\in\R_+$, called the \emph{exponential shear parameter} of $m$ along $\la_i$, obtained as follows: let $\tilde{g}_i$ be a lift of $g_i$ to the universal cover of $(S,m)$, which we identify with the upper half-space $\HH^2$. $\tilde{g}_i$ separates two triangles $\tilde{T}^1_i$ and $\tilde{T}^2_i$, bounded by lifts of edges of $\la$ so that the union $\tilde{Q}_i=\tilde{g}_i\cup\tilde{T}^1_i\cup\tilde{T}^2_i$ forms a square in $\HH^2$ with vertices on the real line bounding $\HH^2$. For a given orientation of $\tilde{g}_i$, we name the vertices of $\tilde{Q}_i$ by $z_-$, $z_+$, $z_r$, $z_l$, such that $\tilde{g}_i$ goes from $z_-$ to $z_+$, and $z_r$ and $z_l$ are respectively to the right and to the left of $\tilde{g}_i$. The exponential shear parameter of $m$ along $\la_i$ is defined as
\[x_i=-\text{cross-ratio}(z_r,z_l,z_-,z_+)=-\frac{(z_r-z_-)(z_l-z_+)}{(z_r-z_+)(z_l-z_-)}\in\R_+.\]
Geometrically, $\log x_i$ corresponds to the (signed) distance between the orthogonal projections of $z_r$ and $z_l$ onto $\tilde{g}_i$.

Conversely, one can construct a (possibly incomplete) hyperbolic metric $m$ from any choice of parameters $x_1,\ldots,x_n\in\R_+$ associated to the edges of $\la$, obtained by gluing ideal hyperbolic triangles into squares whose vertices have the prescribed cross-ratio. Its completion is a hyperbolic surface $S'$ with geodesic boundaries and cusps, for which each end of the edges of $\la$ either converges to a cusp or spirals around a geodesic boundary. The direction of the spiraling provides an orientation of the geodesic boundary. This metric on $S'$ admits a unique extension to a complete metric on $S$ whose convex core is $S'$.  In this way, one obtains a homeomorphism $\phi_\la\colon\widetilde{\T}(S)\to\R^n_+$ for every ideal triangulation $\la$ of $S$.

Given $x_1,\ldots,x_n$ the shear parameters of $m\in\widetilde{\T}(S)$ associated to a triangulation $\la$, one can read the geometry of $(S,m)$ around each puncture $v_j$ of $S$ as follows: let $p_j=x^{k_{1j}}_1\cdots x^{k_{nj}}_n$ where $k_{ij}\in\left\{0,1,2\right\} $ is the number of ends of the edge $\la_i$ that converge to $v_j$. Then, if $p_j=1$, $v_j$ is a cusp. Otherwise, $\left|\log p_j\right|$ is the length of the boundary component of $Conv(S,m)$ facing $v_j$. Its sign corresponds to the orientation of the boundary component with respect to the orientation of $Conv(S,m)$. As a consequence, the homeomorphism $\phi_\la$ restricts to an embedding of $\T(S)$ into $\R^n_+$, corresponding to setting all the $p_j$ equal to 1.

The shear parameters are defined along each edge of $\la$ or equivalently for pairs of adjacent triangles. This definition generalizes to any pair of triangles in the universal cover of $S$ as follows: let $\tilde{\la}$ be the lift of $\la$ to the universal cover $\widetilde{S}$ of $(S,m)$. Let $P$ and $Q$ be two ideal triangles in $\widetilde{S}$ delimited by $\tilde{\la}$. Let $\tilde{\la}_{i_1},\ldots,\tilde{\la}_{i_l}$, lifts of $\la_{i_1},\ldots,\la_{i_l}$ respectively, be the set of edges of $\tilde{\la}$ separating $P$ and $Q$. We include in this set the edges of $P$ and $Q$ which are closest to each other. The \emph{shearing cocycle} $\s$ of $m\in\T(S)$ associated to $\la$ is defined for such triangles by
\[\s(P,Q)=\sum^l_{j=1}\log x_{i_j}
\]
where $x_{i_j}$ is the shearing parameter of $m$ for the edge $\la_{i_j}$. Some properties of the shearing cocycles will be needed later on and we refer to \cite{Bo1} for more details.

\subsection{The Weil--Petersson Poisson structure\label{WP}}
The enhanced Teichm\"uller space can be endowed with a Poisson structure which admits a simple expression in the logarithmic shear coordinates associated to a triangulation $\la$ of $S$ (see for example \cite{Fo}). $S\smallsetminus\la$ has $2n$ spikes converging toward the punctures, each of them delimited by edges $\la_i$ and $\la_j$ not necessarily distinct. For $i$, $j\in\left\{1,\ldots,n\right\}$, let $a_{ij}\in\left\{0,1,2\right\}$ be the number of spikes of $S\smallsetminus\la$ which are delimited to the left by $\la_i$ and to the right by $\la_j$, when looking toward the end of the spikes. The Weil--Petersson Poisson structure is given in coordinates by the following bi-vector
\[\Pi_{WP}=\sum_{i,j}\sigma_{ij}\frac{\partial}{\partial\log x_i}\wedge\frac{\partial}{\partial \log x_j}
\]
where 
\[\sigma_{ij}=a_{ij}-a_{ji}\in\left\{-2,-1,0,1,2\right\}.
\]

Its degeneracy is well understood: The lengths $\log p_j$ associated to the punctures $v_j$ are Casimir functions for $\Pi_{WP}$ and the cusped Teichm\"uller space $\T(S)=\left\{\log p_1=\ldots=\log p_s=0\right\}$ is a symplectic leaf for this structure. The induced symplectic form on $\T(S)$ can then be identified with the K\"ahler form associated to the usual Weil--Petersson metric on Teichm\"uller space.



\section{Pinching along curves: geometric aspects}

We suppose once again that $S$ is an oriented surface of genus $g$ with $s\geq 1$ punctures $v_1$, \ldots, $v_s$ and such that $\chi(S)<0$. Let $\g=\g_1\cup\cdots\cup\g_k$ be a multicurve and $S_\g=S\smallsetminus\g$. It is homeomorphic to a surface with $s+2k$ punctures: the ``old'' ones $v_1$, \ldots, $v_s$ and two new punctures $v'_i$ and $v''_i$ corresponding to the removal of $\g_i$ for $i=1,\ldots k$. Alternatively one can think of $S_\g$ as being obtained from $S$ by pinching the multicurve $\g$ to $k$ nodes and removing them. Note that $S_\g$ may be disconnected.

The goal of this section is to describe explicitly the behavior of the shear coordinates and the Weil--Petersson Poisson structure when going from $\T(S)$ to $\T(S_\g)$ in the topology of the augmented Teichm\"uller space $\overline{\T(S)}$.

\subsection{Induced ideal triangulations}

Given an ideal triangulation $\lambda=\left\{\la_1,\ldots,\la_n\right\}$ of $S$, we want to define an induced ideal triangulation $\lambda_\gamma$ of $S_\gamma$. We can choose $\g=\cup_i\g_i$ so that the $\g_i$ never cross the same edge twice in a row. Let $\lambda\smallsetminus\gamma$ denote the family of arcs obtained from the edges of $\lambda$ intersected with $S_\gamma$. We group these arcs into distinct isotopy classes $\mu_1$, \ldots, $\mu_l$ in $S_\g$, each consisting of a certain number of segments of $\la_1$, \ldots, $\la_n$. Hence $\left\{\mu_1,\ldots,\mu_l\right\}$ is a family of non-intersecting and non-homotopic arcs in $S_\g$.

\begin{lem}
\label{inducedtriang}
The family $\lambda_\gamma=\left\{\mu_1,\ldots,\mu_n\right\}$ is an ideal triangulation of $S_\gamma$ consisting of $n$ distinct isotopy classes of arcs.
\end{lem}

\begin{proof} Since the $\g_i$ do not backtrack, $\lambda\smallsetminus\gamma$ decomposes $S_\gamma$ into pieces of the form given in Figure~\ref{pieces}, where the dashed lines represent $\gamma$ and the first piece can have 0, 1, 2 or 3 such sides.
\begin{figure}[htb!]
\includegraphics{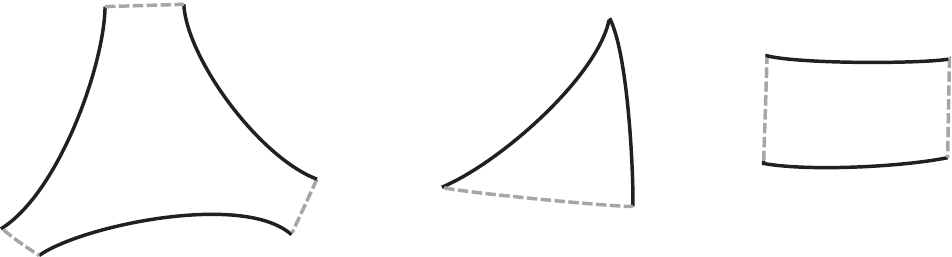}
\caption{\label{pieces}}
\end{figure}
The first piece is an ideal triangle in $S_\g$ and the other two are bigons. One can collapse these bigons successively to arcs with one or two vertices being among the new punctures of $S_\g$. This can be done for all the bigons successively unless one of the situations described in Figure~\ref{problem} occurs. This cannot happen, however, since the $\gamma_i$ are essential and non-homotopic to each other.
\begin{figure}[htb!]
\includegraphics{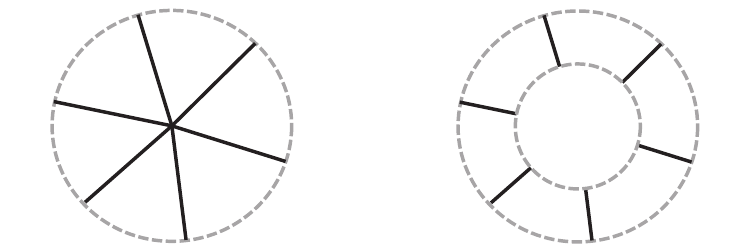}
\caption{\label{problem}}
\end{figure}
The remaining arcs correspond to the homotopy classes $\mu_1,\ldots,\mu_l$ and decompose $S_\g$ into ideal triangles. Since $\chi(S_\g)=\chi(S)$ and the number of edges of an ideal triangulation depends only on the Euler characteristic, the ideal triangulation $\la_\g=\left\{\mu_1,\ldots,\mu_n\right\}$ of $S_\g$ has the same number of edges as $\la$.
\end{proof}

We call $\la_\g$ the \emph{ideal triangulation of $S_\g$ induced by $\la$}.

\begin{rem}
\label{process}
In practice each edge $\mu_i$ of $\la_\g$ is obtained from $\la$ by considering a maximal sequence of adjacent bigons in the decomposition of $S_\g$ by $\la\smallsetminus\g$ and collapsing it to an edge. Via this process, ideal triangles for $\la_\g$ on $S_\g$ are identified naturally with ideal triangles for $\la$ on $S$ (see Figure~\ref{induced}).

\begin{figure}[htb!]
\includegraphics{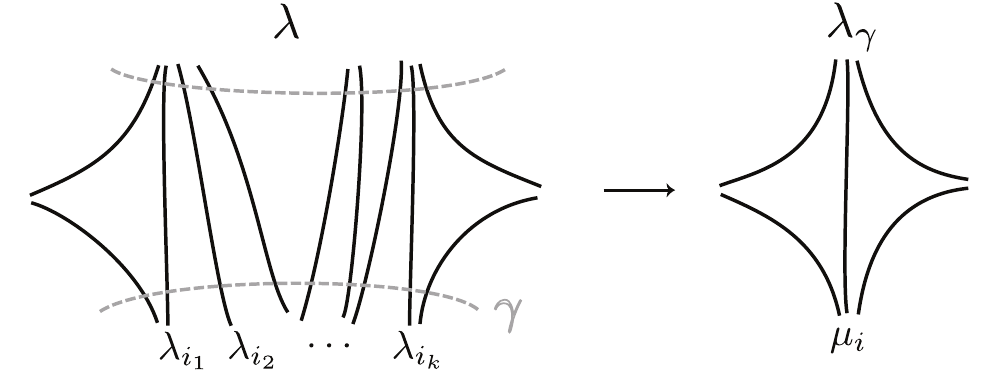}
\caption{\label{induced}}
\end{figure}

\end{rem}

\subsection{Extension of the shear coordinates\label{extension}}

Let $\la=\left\{\la_1,\ldots,\la_n\right\}$ be an ideal triangulation of $S$ and $\la_\g=\left\{\mu_1,\ldots,\mu_n\right\}$ be the induced triangulation of $S_\g$. We suppose that, for $i=1,\ldots,n$, $\mu_i$ corresponds to the homotopy class of  $k_{ij}$ segments from $\la_j$ for $j=1,\ldots,n$. The following proposition relates the shear coordinates on $\T(S)$ associated to $\la$ to the ones on $\T(S_\g)$ associated to $\la_\g$.

\begin{prop}
\label{coordlimit} 
Let $m_t\in\T(S)$ be a continuous family of hyperbolic metrics on $S$, $(x_1(t),\ldots,x_n(t))\in\R^n_+$ their shear parameters for $\la$, and  $m_\g\in\T(S_\gamma)$, $(y_1,\ldots,y_n)\in\R^n_+$ its shear parameters for $\la_\g$. Then
\begin{align*}
m_t \xrightarrow[t\rightarrow 0]{} m_\g\ \text{in}\ \overline{\mathcal{T}(S)}\Rightarrow \lim_{t\rightarrow 0} x^{k_{i1}}_1(t)\cdots x^{k_{in}}_n(t)=y_i\ \text{for}\ i=1,\ldots,n.
\end{align*}
\end{prop}

\begin{proof} Let $\mu_i$ be an edge of $\la_\g$ in $S_\g$ and $\tilde{\mu}_i$ be one of its lifts to the universal cover of $(S_\g,m_\g)$. Then $\tilde{\mu}_i$ is the diagonal of a square consisting of two ideal triangles $P^i_0$ and $Q^i_0$. Note that the universal cover of $(S_\g,m_\g)$ is isometric to several copies of $\HH^2$, one for each connected component of $S_\g$. We denote by $\widetilde{S}_\g$ the one containing $\tilde{\mu}_i$. By Remark~\ref{process}, $\mu_i$ corresponds to a rectangle composed of a succession of bigons in the decomposition of $S_\g$ by $\la\smallsetminus\g$, ending at the sides of two (non-necessarily distinct) triangles. We consider a lift of this rectangle to the universal cover $\widetilde{S_t}$ of $(S,m_t)$. It ends at the sides of two triangles $P^i_t$ and $Q^i_t$ which are separated by $k_{ij}$ lifts of the edge $\la_j$ for $j=1,\ldots,n$. Hence the shearing cocycle $\sigma_t$ associated to $m_t$ satisfies
\[\sigma_t(P^i_t,Q^i_t)=\sum_j k_{ij}\log x_j(t).
\]

In addition, since $m_t\rightarrow m_\g$, these lifts can be chosen so that, with the right identification of $\widetilde{S}_\g$ with $\widetilde{S_t}$, $P^i_t$ and $Q^i_t$ approach $P^i_0$ and $Q^i_0$ respectively as $t\rightarrow 0$.

We can then use the following inequality, derived from Lemma 8 in \cite{Bo1}: if $a_t$ (resp. $b_t$) is the projection of the third vertex of $P^i_t$ (resp. $Q^i_t$) onto $g_t$ (resp. $h_t$),where $g_t$ and $h_t$ are the edges of $P^i_t$ and $Q^i_t$ which are the closest to each other, and if $b'_t$ is the projection of $b_t$ onto $g_t$, then
\[\left|\sigma_t(P^i_t,Q^i_t)-d(a_t,b'_t)\right|\leq l_{m_t}(\gamma).
\]
We refer to Figure~\ref{shear} for an example with notations.
\begin{figure}[htb!]
\includegraphics{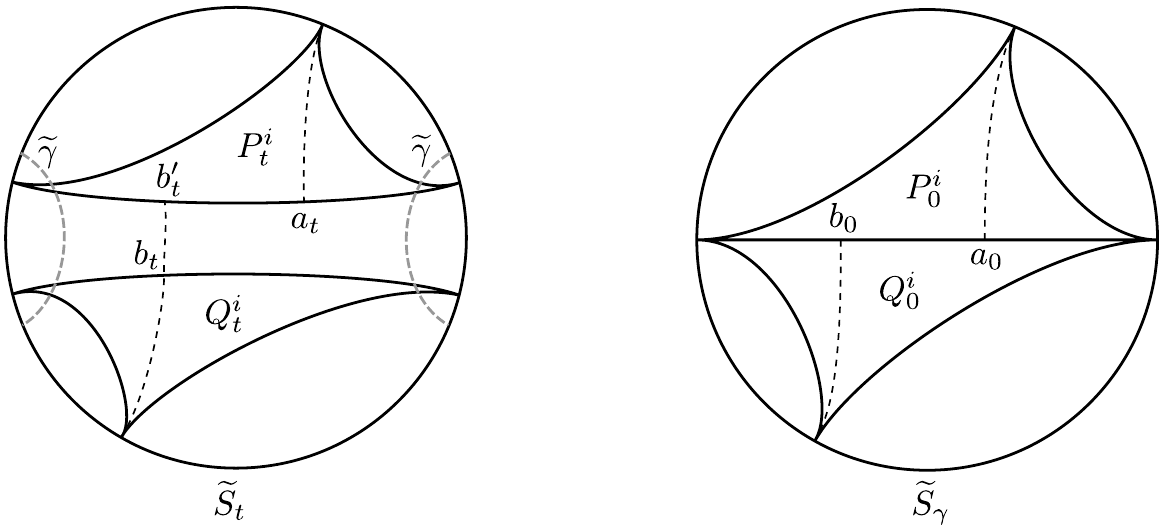}
\caption{\label{shear}}
\end{figure}

Since $P^i_t$ and $Q^i_t$ approach $P^i_0$ and $Q^i_0$ respectively, we have $a_t\rightarrow a_0$ and $b_t,b'_t\rightarrow b_0$ where $a_0$ and $b_0$ are the projections of the third vertex of $P^i_0$ and $Q^i_0$, respectively, onto $g_0=h_0=\tilde{\mu}_i$. Then
\begin{align*}
\lim_{t\rightarrow 0}\sum_j k_{ij}\log x_j(t)=\lim_{t\rightarrow 0}\sigma_t(P^i_t,Q^i_t)&=\lim_{t\rightarrow 0}d(a_t,b'_t)\\
											 &=d(a_0,b_0)\\
											 &=\sigma_0(P^i_0,Q^i_0)=\log y_i.
\end{align*}
\end{proof}

In other words, the shear parameter on $\T(S_\g)$ associated to an edge $\mu_i$ of $\la_\g$ is the limit of a monomial in the shear parameters on $\T(S)$ associated to $\la$. This monomial is given by the product of parameters associated to the (segments of) edges of $\la$ constituting the homotopy class of $\mu_i$.

Another important monomial can be associated to $\g$ itself, when it consists of one simple closed curve. We denote by $c_i$ the number of times $\g$ intersects the edge $\la_i$ and let
\[x_\g=x^{c_1}_1\cdots x^{c_n}_n.\]
We call this monomial the \emph{exponential graph length} of $\g$.

If $m_t\in\T(S)$, $t\in(0,1]$, is a family of hyperbolic metrics with shear parameters $x_1(t)$, \ldots, $x_n(t)$ and $x_\g(t)$ is the associated exponential graph length of $\g$, we have the following lemma.

\begin{lem}
\label{limitlength}
If, as $t\rightarrow 0$, $l_{m_t}(\g)$ approaches 0 then the exponential graph length $x_\g(t)$ of $\g$ approaches 1.
\end{lem}
\begin{proof}

This is a consequence of the formula for the length of $\g$ in shear coordinates as described for example in \cite{CheFo}. In particular, we have
\begin{align}
\label{comblength}
2\cosh(l_{m_t}(\g)/2)=x^{1/2}_\g(t)+x^{-1/2}_\g(t)+\ldots
\end{align}
where all the terms in the sum are positive.

Since the left hand side approaches 2, we see that $x_\g(t)$ must approach 1 as $t\rightarrow 0$.
\end{proof}

It also follows from this proof that the other terms on the right hand side of (\ref{comblength}) approach 0 as $t\rightarrow 0$. This means that, near the stratum $\T(S_\g)$, the length function $l_{.}(\g)$ is asymptotically equivalent to the graph length function $\log x_\g$. This explains the essential r\^ole the quantum analogue $X_\g$ of $x_\g$ will play later on.

\subsection{Extension of the Weil--Petersson Poisson structure}

Masur \cite{Mas1} (see also Wolpert \cite{Wol}) proved that the Weil--Petersson K\"ahler metric on $\T(S)$ extends in an appropriate sense to its augmentation $\overline{\T(S)}$ and can be identified with the Weil--Petersson metric on the lower dimensional strata $\T(S_\g)$. We would like to know how this fact together with Proposition~\ref{coordlimit} translate in terms of the expression of the Weil--Petersson Poisson structures on $\T(S)$ and $\T(S_\g)$ in the shear coordinates associated to $\la$ and $\la_\g$ respectively. To do so we use Lemma \ref{homological} given below, which is a homological interpretation of the Weil--Petersson structure as described for example in \cite{BoLiu}. Proposition~\ref{WPcoeff} can then be interpreted as a topological translation of the result of Masur.

Let $\s=(\s_{ij})_{ij}$ be the matrix of coefficients of the Weil--Petersson Poisson structure on $\widetilde{\T}(S)$ in the coordinates $(\log x_1,\ldots,\log x_n)$ associated to an ideal triangulation $\la=\left\{\la_1,\ldots,\la_n\right\}$, as was described in Section~\ref{WP}. Setting $\s(\la_i,\la_j)=\s_{ij}$, $\s$ can be identified with an antisymmetric bilinear form on $\h(\la,\Z)\cong\Z^n$, the free abelian group generated over the set of edges of $\la$. We are going to use a homological interpretation of $\s$ as given for example in \cite{BoLiu}. This formulation follows \cite{Bo1} where it is used to describe the Thurston symplectic form.

Let $G$ be the dual graph of $\la$ and $\widehat{G}$ be the oriented graph obtained from $G$ by keeping the same vertex set and replacing each edge of $G$ by two oriented edges which have the same endpoints as the original edge but with opposite orientations.

There is a unique way to thicken $\widehat{G}$ into a surface $\widehat{S}$ such that:
\begin{enumerate}
\item $\widehat{S}$ deformation retracts to $\widehat{G}$;
\item as one goes around a vertex $\widehat{v}$ of $\widehat{G}$ in $\widehat{S}$, the orientation of the edges of $\widehat{G}$ ending at $\widehat{v}$ points alternatively toward and away from $\widehat{v}$;
\item the natural projection $p\colon\widehat{G}\to G$ extends to a 2-fold cover $\widehat{S}\to S$, branched along the vertex set of $\widehat{G}$.
\end{enumerate}

Let $\eta\colon\widehat{S}\to\widehat{S}$ be the covering involution of the branched cover $p\colon\widehat{S}\to S$.

\begin{lem}
\label{homological}
The group $\h(\la,\Z)$ can be identified with the subgroup of $H_1(\widehat{S})$ consisting of those $\widehat{\alpha}$ such that $\eta_{*}(\widehat{\alpha})=-\widehat{\alpha}$. In addition, if $\alpha$, $\beta\in\h(\la,\Z)$ correspond to $\widehat{\alpha}$, $\widehat{\beta}\in H_1(\widehat{S})$, then $\s(\alpha,\beta)=\widehat{\alpha}\cdot\widehat{\beta}$, their algebraic intersection number.
\end{lem}

\begin{proof}
cf. Lemma 6 and 7 in \cite{BoLiu}.
\end{proof}

The identification of Lemma~\ref{homological} is given as follows: if $e_i\in\h(\la,\Z)$ is the element associating weight 1 to the edge $\la_i$ and weight 0 to the other edges, $\widehat{e}_i$ is the lift in $\widehat{G}$ of the edge of $G$ dual to $\la_i$. The closed curve $\widehat{e}_i$ comes with a natural orientation given by the one on $\widehat{G}$, and we identify $\widehat{e}_i$ with its homology class in $H_1(\widehat{S})$. More generally, to $\alpha=\sum\alpha_i e_i$, we can then associate the homology class $\widehat{\alpha}=\sum\alpha_i \widehat{e}_i$.

Suppose now that $\tau=(\tau_{ij})_{ij}$ is the matrix of coefficients of the Weil--Petersson Poisson structure on $\widetilde{\T}(S_\g)$ for the induced ideal triangulation $\la_\g=\left\{\mu_1,\ldots,\mu_n\right\}$. We suppose that, for $i=1,\ldots,n$, $\mu_i$ corresponds to the homotopy class of  $k_{ij}$ segments from $\la_j$ for $j=1,\ldots,n$. We let $\mathbf{k}_i=(k_{i1},\ldots,k_{in})\in\Z^n$ and identify $\s$ with a bilinear form on $\Z^n\cong\h(\la,\Z)$. The following proposition relates the entries of $\tau$ and $\s$.

\begin{prop}
\label{WPcoeff}
With the notations above, the coefficients of the Weil--Petersson Poisson structure on $\T(S)$ and $\T(S_\g)$ for $\la$ and $\la_\g$ respectively are related via the following formula: for $i,j=1\ldots n$,
\begin{align*}
\tau_{ij}=\s(\mathbf{k}_i,\mathbf{k}_j)=\sum_{s,t}k_{is}k_{jt}\s_{st}.
\end{align*}
\end{prop}

\begin{proof}
Let $G$ be the graph dual to $\la$ in $S$, $G_\g$ be the graph dual to $\la_\g$ and $G'_\g$ the graph dual to $\la\smallsetminus\g$ in $S_\g$. We recall that $\la\smallsetminus\g$ denotes the family of arcs obtained from the edges of $\la$ intersected with $S_\g$, where we \emph{do not} consider the arcs up to homotopy. It decomposes $S_\g$ into triangles and bigons, hence the vertices of $G'_\g$ are either bivalent or trivalent. We will call a \emph{maximal chain} in $G'_\g$ any chain of edges connected via bi-valent vertices and with endpoints at trivalent vertices. In particular, edges connecting trivalent vertices \emph{are} maximal chains. Following Remark~\ref{process}, the maximal chains of $G'_\g$ are in on-to-one correspondence with the edges of $G_\g$, and this defines a natural homeomorphism $G'_\g\cong G_\g$. In addition, the identification of ideal triangles for $\la$ and $\la_\g$ gives an identification of $G$ and $G'_\g$ in a neighborhood of each of their trivalent vertices.

We also consider the oriented graphs $\widehat{G}$ and $\widehat{G}_\g$ (see beginning of section), as well as the oriented graph $\widehat{G}'_\g$ obtained from $G'_\g$ by keeping the same set of trivalent vertices and replacing each maximal chain by two such chains connected to the same (6-valent) endpoints, endowed with opposite orientations. This graph is naturally homeomorphic to $\widehat{G}_\g$. As described above, $\widehat{G}$ thickens into $\widehat{S}$, and both $\widehat{G}_\g$ and $\widehat{G}'_\g$ thicken into $\widehat{S}_\g$. We have covering maps $p\colon\widehat{S}\rightarrow S$ and $p_\g\colon\widehat{S}_\g\rightarrow S_\g$ which restrict to the corresponding graphs.

Recall that, by definition, $S_\g=S\smallsetminus\g\subset S$. Similarly one can identify $\widehat{S}_\g$ with $\widehat{S}\smallsetminus\widehat{\g}$ where $\widehat{\g}=p^{-1}(\g)$. Indeed, let $U\subset S_\g\subset S$ be a union of small discs around the trivalent vertices of $G'_\g$ where it is identified with $G$. By construction, we have $p^{-1}_\g(U)\cong p^{-1}(U)\subset\widehat{S}$. Outside of $U$ both coverings are trivial, so we also have a natural identification $p^{-1}_\g(S_\g\smallsetminus U)\cong p^{-1}(S_\g\smallsetminus U)\subset\widehat{S}$. Hence we obtain that $\widehat{S}_\g\cong p^{-1}(S_\g)\subset\widehat{S}$. Note, in addition, that $\g$ can be chosen so that it doesn't pass through any of the ramification points of $p$ (that is, the vertices of $G$). With this assumption, $p^{-1}(S_\g)=\widehat{S}\smallsetminus\widehat{\g}$, where $p^{-1}(\g)=\widehat{\g}$ consists of two non intersecting multicurves and we can identify $\widehat{S}_\g$ with $\widehat{S}\smallsetminus\widehat{\g}$ sitting in $\widehat{S}$. Accordingly, $p_\g$ is identified with the restriction of $p$ to $S\smallsetminus\g$.

The inclusion $\widehat{\iota}\colon\widehat{S}_\g\hookrightarrow\widehat{S}$ induces a map $\widehat{\iota}_*\colon H_1(\widehat{S}_\g)\to H_1(\widehat{S})$ at the level of homology. By construction, if we denote by $\eta_\g$ the covering involution associated to $p_\g$, we have that $\widehat{\iota}_*\circ\eta_{\g*}=\eta_*\circ\widehat{\iota}_*$.

On the other hand, there is a natural map $\pi\colon G'_\g\to G$ defined by sending each edge of $G'_\g$ onto the edge of $G$ dual to the same edge of the triangulation $\la$, which lifts to a $\eta$-invariant map $\widehat{\pi}\colon\widehat{G}'_\g\to\widehat{G}$. One can then consider retractions $\widehat{r}\colon\widehat{S}\to\widehat{G}$ and $\widehat{r}_\g\colon\widehat{S}_\g\to\widehat{G}'_\g$ such that $\widehat{\pi}\circ\widehat{r}_\g$ is equal to $\widehat{r}\circ\widehat{\iota}$, giving the following commutative diagram:
\begin{align}
\label{diagram1}
\xymatrix{{H_1(\widehat{S}_\g)} \ar[r]^{\widehat{\iota}_*} \ar[d]^{\widehat{r}_{\g *}}_{\wr} & {H_1(\widehat{S})} \ar[d]^{\widehat{r}_*}_{\wr}\\
					{H_1(\widehat{G}'_\g)} \ar[r]^{\widehat{\pi}_*} & {H_1(\widehat{G})}.}
\end{align}

Let $e_i$ be the generator of $\h(\la,\Z)$ assigning weight 1 to $\la_i$ and 0 to the other edges of $\la$, and $f_i$ be the generator of $\h(\la_\g,\Z)$ assigning weight 1 to $\mu_i$ and 0 to the other edges of $\la_\g$. Following Lemma~\ref{homological}, we associate to $e_i$ the homology class $\widehat{e}_i\in H_1(\widehat{S})$ corresponding to the lift in $\widehat{G}$ of the edge of $G$ dual to $\la_i$. To $f_i$ on the other hand, we associate $\widehat{f}_i\in H_1(\widehat{S}_\g)$ which corresponds to the lift in $\widehat{G}'_\g$ of the maximal chain of $G'_\g$ dual to $\mu_i$. By construction, this chain consists of the lift of $k_{is}$ edges dual to $\la_s$ for $s=1,\ldots,n$. Hence, as curves, $\widehat{\pi}(\widehat{f}_i)$ covers $k_{is}$ times each $\widehat{e}_s$. Homologically, using the commutativity of diagram (\ref{diagram1}), we obtain
\begin{align}
\label{pi}
\widehat{\iota}_*(\widehat{f}_i)=\widehat{r}^{-1}_*\circ\widehat{\pi}_*\circ\widehat{r}_{\g*}(\widehat{f}_i)=\sum^n_{s=1} k_{is}\widehat{e}_s\text{ for }i=1,\ldots,n.
\end{align}

Hence we have the following equalities:
\begin{align*}
\tau_{ij}=\widehat{f}_i\cdot\widehat{f}_j&=\widehat{\iota}_*(\widehat{f}_i)\cdot\widehat{\iota}_*(\widehat{f}_j)\\
&=\sum_{s}k_{is}\widehat{e}_s\cdot\sum_{t}k_{jt}\widehat{e}_t\\
&=\sum_{s,t}k_{is}k_{jt}\widehat{e}_s\cdot\widehat{e}_t=\sum_{s,t}k_{is}k_{jt}\s_{st}\,.
\end{align*}

\end{proof}

In terms of shear coordinates and with the notations of Section \ref{extension}, Proposition \ref{WPcoeff} implies that the Poisson brackets associated to the Weil--Petersson structure on $\T(S)$ and $\T(S_\g)$ are related via the formula
\[\left\{\log y_i,\log y_j\right\}_{\T(S_\g)}=\left\{\log(x^{k_{i1}}_1\cdots x^{k_{in}}_n),\log(x^{k_{j1}}_1\cdots x^{k_{jn}}_n)\right\}_{\T(S)},\]
which is consistent with Proposition \ref{coordlimit} and \cite{Mas1,Wol}.



\section{Pinching along curves: quantum aspect}

\subsection{The Chekhov--Fock algebra}

Let $\la=\left\{\la_1,\ldots,\la_n\right\}$ be an ideal triangulation of $S$ and fix a non-zero complex number $q\in\C^*$. Following \cite{Liu1}, we define the \emph{Chekhov--Fock algebra} $\T^q_\la(S)$ of $S$ associated to $\la$ to be the algebra over $\C$ with generators $X^{\pm 1}_i$ associated to the edges $\la_i$ of $\la$ and subject to the relations
\[X_i X_j=q^{2\s_{ij}}X_j X_i\]
for every $i$, $j$, where the $\s_{ij}\in\left\{-2,-1,0,1,2\right\}$ are the coefficients of the Weil--Petersson Poisson structure on $\widetilde{\T}(S)$ in the shear coordinates associated to $\la$. We will sometimes use the notation $\T^q_\la(S)=\C\left[X_1,\ldots,X_n\right]^q_\la$ to specify the generators of the algebra.

If $A$ and $B$ are two monomials in the variables $X_1$, \ldots, $X_n$, then they satisfy a relation of the form $AB=q^{2\alpha}BA$ for some integer $\alpha$, and we will use the notation $\s(A,B)=\alpha$. This coefficient is independent of the order of the generators inside each monomial.

For $\mathbf{k}=(k_1,\ldots,k_n)\in\Z^n$, if $A$ is a monomial consisting of $k_i$ times the generator $X_i$ for i=$1,\ldots,n$, in any given order, we define the following element in $\T^q_\la(S)$:
\[\left[A\right]=X_{\mathbf{k}}=q^{-\sum_{i<j}k_i k_j\s_{ij}}X^{k_1}_1\cdots X^{k_n}_n.
\]
This is known as the \emph{Weyl quantum ordering}. These monomials satisfy the following relations:
\begin{align*}
X_\mathbf{k}X_\mathbf{l}&=q^{\s(\mathbf{k},\mathbf{l})}X_\mathbf{k+l}\\
                        &=q^{2\s(\mathbf{k},\mathbf{l})}X_\mathbf{l}X_\mathbf{k},
\end{align*}
where we once again identify $\s$ with a bilinear form on $\Z^n$. The different notations for $\s$ coincide in the sense that $\s(X_\mathbf{k},X_\mathbf{l})=\s(\mathbf{k},\mathbf{l})$.

In particular, if $\alpha$ is a path between two vertices in the dual graph $G$ of $\la$ (which does not backtrack), we can identify it with the element $\alpha=(\alpha_1,\ldots,\alpha_n)$ of $\h(\la,\Z)$, where $\alpha_i$ is the number of times the path $\alpha$ passes through the edge of $G$ dual to $\la_i$. Then we associate to $\alpha$ the monomial $X_\alpha$ defined as above. If $\beta$ is another path in $G$, we have
\[
X_\alpha X_\beta=q^{2\widehat{\alpha}\cdot\widehat{\beta}}X_\beta X_\alpha
\]
by Lemma~\ref{homological}, where $\widehat{\alpha}$ and $\widehat{\beta}$ are the associated elements in $H_1(\widehat{S})$. Of particular interest will be the element $X_\g$ associated to a simple closed curve $\g$ in $S$, which we identify with its retraction to a cycle in $G$.

If $\Sigma$ is another surface with ideal triangulation $\mu$, a homomorphism between $\T^q_\la(S)$ and $\T^q_\mu(\Sigma)$ doesn't in general preserve the quantum ordering. However, we have the following elementary lemma which will be useful later on.

\begin{lem}
\label{quantumorder}
Let $A_1$, \ldots, $A_s$ be monomials in $\T^q_\la(S)$, $B_1$, \ldots, $B_s$ be  monomials in $\T^q_\mu(\Sigma)$. If $\Psi\colon\T^q_\la(S)\to\T^q_\mu(\Sigma)$ is an algebra homomorphism such that $\Psi(\left[A_i\right])=\left[B_i\right]$ for all $i$, then $\Psi(\left[A_1\cdots A_s\right])=\left[B_1\cdots B_s\right]$.
\end{lem}

\subsection{A homomorphism between Chekhov--Fock algebras}

As a direct consequence of Proposition~\ref{coordlimit} and \ref{WPcoeff}, we construct a natural homomorphism between the Chekhov--Fock algebras associated to $S$ and $S_\g$. We recall that, by lemma~\ref{inducedtriang}, $\la$ induces an ideal triangulation $\la_\g=\left\{\mu_1,\ldots,\mu_n\right\}$ of $S_\g$, where $\mu_i$ is the homotopy class in $S_\g$ of $k_{ij}$ segments from $\la_j$, for $j=1,\ldots,n$. We let $\mathbf{k}_i=(k_{i1},\ldots,k_{in})$ for $i=1,\ldots,n$.

\begin{prop}
\label{homom}
The map
\[\Theta^q_{\gamma,\lambda}\colon\mathcal{T}^q_{\lambda_\gamma}(S_\gamma)=\mathbb{C}\left[Y_1,\ldots,Y_n\right]^q_{\lambda_\gamma}\longrightarrow \mathbb{C}\left[X_1,\ldots,X_n\right]^q_\lambda=\mathcal{T}^q_{\lambda}(S)
\]
defined on the generators by 
\[\Theta^q_{\gamma,\lambda}(Y_i)=X_{\mathbf{k}_i}
\]
extends to an algebra homomorphism.
\end{prop}

\begin{proof}
We check it on the generators of $\T^q_{\la_\g}(S_\g)$:
\begin{align*}
\Theta^q_{\gamma,\lambda}(Y_i Y_j)&=\Theta^q_{\gamma,\lambda}(Y_i)\Theta^q_{\gamma,\lambda}(Y_j)\\
    &=X_{\mathbf{k}_i}X_{\mathbf{k}_j}\\
    &=q^{2\s(\mathbf{k}_i,\mathbf{k}_j)}X_{\mathbf{k}_j}X_{\mathbf{k}_i}\\
    &=q^{2\tau_{ij}}\Theta^q_{\gamma,\lambda}(Y_j Y_i),
\end{align*}
where the last equality is given by Proposition~\ref{WPcoeff}.
\end{proof}

The following lemma states that, if we pinch the curves constituting $\g$ in different orders, the resulting homomorphisms given by Proposition~\ref{homom} are the same.

\begin{lem}
\label{composition}
Consider any sequence of integers $i_1$, \ldots, $i_k$ such that\\ $\{i_1,\ldots,i_k\}=\{1,\ldots,k\}$ and let $\g^l=\cup^l_{j=1}\g_{i_j}$ for $l\leq k$. Then
\[
\Theta^q_{\g,\la}=\Theta^q_{\g_{i_k},\la_{\g^{k-1}}}\circ\Theta^q_{\g_{i_{k-1}},\la_{\g^{k-2}}}\circ\cdots\circ\Theta^q_{\g_{i_2},\la_{\g^1}}\circ\Theta^q_{\g_{i_1},\la}.
\]
\end{lem}
\begin{proof} The definition of the edges $\mu_i$ of $\la_\g$ as homotopy classes of segments from $\la_1$, \ldots, $\la_n$ does not depend on the order in which the curves $\g_1$, \ldots, $\g_k$ are pinched. Hence the generators $Y_i$ of $\T^q_\la(S_\g)$ are sent to the same monomials by the two maps, \emph{up to ordering}. Since all the maps considered send generators to quantum ordered monomials, Lemma~\ref{quantumorder} implies that the quantum orders are respected on each side and hence the maps coincide.
\end{proof} 

\begin{rem}
\label{theta}
For future reference, we want to interpret $\Theta^q_{\g,\la}$ in terms of dual graphs. As in the proof of Proposition~\ref{WPcoeff}, we let $G$ be the graph dual to $\la$ in $S$ and $G'_\g$ be the graph dual to $\la\smallsetminus\g$ in $S_\g$. There is a natural map $\pi\colon G'_\g\to G$ defined by sending each edge of $G'_\g$ onto the edge of $G$ dual to the same edge of $\la$. If $\alpha$ is any path between trivalent vertices in $G'_\g$, we denote by $Y_\alpha\in\T^q_{\la_\g}(S_\g)$ the quantum  ordered product of generators associated to the edges crossed by $\alpha$ (after making the identification $G'_\g\cong G_\g$ with the graph dual to $\la_\g$), and we define similarly $X_\beta\in\T^q_\la(S)$ for $\beta$ any path between vertices in $G$. Then, by definition of $\Theta^q_{\g,\la}$, we have
\[\Theta^q_{\g,\la}(Y_\alpha)=X_{\pi(\alpha)}.\]
\end{rem}

\section{Application to the representation theory of Chekhov--Fock algebras}

\subsection{Representations of the Chekhov--Fock algebras}

The irreducible finite dimensional representations of $\T^q_\la(S)$ for $q$ a root of unity have been studied in details in \cite{BoLiu}. We will recall the main results which will be needed for our purpose.

An important step is to describe the center of these algebras. For $i=1,\ldots,s$, consider the element
\[
P_i=X_{\mathbf{p}_i}=\left[X^{p_{i1}}_1 X^{p_{i2}}_2\cdots X^{p_{in}}_n\right]
\]
of $\T^q_\la(S)$ associated to the puncture $v_i$ of $S$, where $p_{ij}\in\left\{0,1,2\right\}$ is the number of ends of the edge $\la_j$ that converge to $v_i$ and $\mathbf{p}_i=(p_{i1},\ldots,p_{in})$. Namely, $P_i$ is the quantum ordered product of generators associated to the edges ending at $v_i$. If $\alpha_i$ is a small loop around $v_i$ and we identify it with its retraction to a cycle in the dual graph $G$, we also have $P_i=X_{\alpha_i}$, with the notations introduced previously.

In addition, let $\mathbf{h}=(1,\ldots,1)$ and
\[H=X_\mathbf{h}=\left[X_1\cdots X_n\right].
\]
The following proposition describes the center of the Chekhov--Fock algebras for certain non-generic values of the parameter $q$.

\begin{prop}[Bonahon--Liu]
If $q^2$ is a primitive $N$--th root of unity with $N$ odd, the center $\mathcal{Z}^q_\la$ of $\T^q_\la(S)$ is generated by the elements $X^N_i$ for $i=1,\ldots,n$, the $P_j$ for $j=1,\ldots,s$ and $H$.
\end{prop}

A complete classification of the finite dimensional irreducible representations of $\T^q_\la(S)$ when $q^2$ is a root of unity was obtained in \cite{BoLiu}. A particular case of this classification can be summarized by the following theorem. Note the further restriction to a condition on $q$ instead of $q^2$.

\begin{thm}[Bonahon--Liu]
\label{BoLiu1}
Suppose that $q$ is a primitive $N$--th root of unity with $N$ odd. Every irreducible finite dimensional representation \\ $\rho\colon\T^q_\la(S)\to\End(V)$ has dimension $N^{3g+s-3}$ and is determined completely by its restriction to the center $\mathcal{Z}^q_\la$ of $\T^q_\la(S)$.

In particular, given $p_j\in\left\{0,\ldots,N-1\right\}$  integers labelling the punctures $v_j$ of $S$ and $m\in\T(S)$ with shear parameters $(x_1,\ldots,x_n)$ associated to $\la$, there is a finite dimensional irreducible representation $\rho\colon\T^q_\la(S)\to\End(V)$ such that:
\begin{itemize}
\item $\rho(X^N_i)=x_i \Id_V$ for $i=1,\ldots,n$;
\item $\rho(P_j)=q^{p_j}\Id_V$ for $j=1,\ldots,s$.
\end{itemize}
These conditions determine $\rho$ uniquely up to isomorphism.
\end{thm}

\begin{proof}
The first part is given by Theorem 20 in  \cite{BoLiu}. The second part is a specialization of Theorem 21 to the case where the $x_i$ are positive real numbers and $x^{p_{j1}}\cdots x^{p_{jn}}=1$ for $j=1,\ldots,s$, corresponding to the shear parameters of a complete hyperbolic metric on $S$. In this case $\rho(H)$ is completely determined by the fact that $\rho(H^N)=\Id_V$ and $H^2=P_1\cdots P_s$.
\end{proof}

We call $\mathbf{p}=(p_1,\ldots,p_s)$ the \emph{weights} of the representation $\rho$ associated to the $s$--tuple of punctures $(v_1,\ldots,v_s)$. Depending on the context we will use the notations $\rho_m$ or $\rho^{\mathbf{P}}_m$ to emphasize the dependence of an irreducible representation on its associated weights and metric.

\begin{rem}
\label{dimension}
Theorem~\ref{BoLiu1} generalizes directly to the case of a disconnected surface $S=S_1\sqcup S_2$, except for the statement about the dimensions. More precisely, if $\la=\la_1\cup\la_2$ is a triangulation of $S$, then there is a natural isomorphism $\T^q_\la(S)\cong\T^q_{\la_1}(S_1)\otimes\T^q_{\la_2}(S_2)$, since the generators associated to $\la_1$ and $\la_2$ commute with each other. Hence an irreducible representation of $\T^q_\la(S)$ is a tensor product of irreducible representations $\rho_1\otimes\rho_2$ of each factor and has dimension $N^{3g_1+s_1-3}\times N^{3g_2+s_2-3}=N^{3g+s-6}$. In particular, given a multicurve $\g=\g_1\cup\cdots\cup\g_k$, the dimension of an irreducible representation of $\T^q_\la(S)$ is $N^k$ times the dimension of an irreducible representation of $\T^q_{\la_\g}(S_\g)$, regardless of the connectivity of $S$ and $S_\g$.
\end{rem}

\subsection{\label{convrep} Convergence of representations in the augmented Teichm\"uller space}

We suppose once again that $q$ is a primitive $N$--th root of unity with $N$ odd. Theorem~\ref{BoLiu1} implies that, to $\mathbf{p}=(p_1,\ldots,p_s)\in\{0,\ldots,N-1\}^s$ and a continuous family of hyperbolic metrics $m_t\in\T(S)$, $t\in\left(0,1\right]$, one can associate a continuous family of irreducible representations $\rho_t\colon\T^q_\la(S)\to\End(V)$ as follows: let $x_1(t)$, \ldots, $x_n(t)$ be the shearing parameters associated to $m_t$ for the triangulation $\la$, and let $\rho_1$ be an irreducible representation classified by $m_1$ and weights $p_1$, \ldots, $p_n$. Then, by Theorem~\ref{BoLiu1}, there are elements $A_1$, \ldots, $A_n$ of $\End(V)$ such that $A^N_i=\Id_V$ and $\rho_1(X_i)=\sqrt[N]{x_i(1)}A_i$, for $i=1,\ldots,n$. One can then construct a family of representations $\rho_t$, defined on the generators by $\rho_t(X_i)=\sqrt[N]{x_i(t)}A_i$. In this way, we obtain a family of representations $\rho_t$ classified by $m_t$, $t\in\left(0,1\right]$, and weights $\mathbf{p}$. It is continuous in the sense that, for any element $X\in\T^q_\la(S)$, $\rho_t(X)$ is a continuous family in $\End(V)$.

By composing with the homomorphism $\Theta^q_{\g,\la}$ from Proposition~\ref{homom}, any representation $\rho$ of $\T^q_\lambda(S)$ gives a representation $\rho\circ\Theta^q_{\g,\la}$ of $\T^q_{\la_\g}(S_\g)$. If $\rho_m$ is an irreducible representation classified by $m\in\T(S)$ and weights $\mathbf{p}$, then, by the dimension count done in Remark~\ref{dimension}, $\rho_m\circ\Theta^q_{\g,\la}$ is a \emph{reducible} representation. We would like to know how this representation decomposes into irreducible subrepresentations when $m$ approaches $m_\g\in\T(S_\g)$.

We recall that the punctures of $S_\g$ are $v_1$, \ldots, $v_s$ corresponding to the same punctures in $S$ together with the new punctures $v'_1$, $v''_1$, \ldots, $v'_k$, $v''_k$ corresponding to the removal of the curves $\g_1$, \ldots, $\g_k$. Given $\mathbf{p}=(p_1,\ldots,p_s)$ weights labelling the punctures of $S$, we say that $\mathbf{p_\g}\in\{0,\ldots,N-1\}^{s+2k}$ are \emph{compatible weights} labelling $(v_1,\ldots,v_s,v'_1,v''_1,\ldots,v'_k,v''_k)$ if they are of the form $\mathbf{p_\g}=(p_1,\ldots,p_s,i_1,i_1,\ldots,i_k,i_k)$.

\begin{thm}
\label{cheforep}
Let $m_t\in\T(S)$ be a continuous family of hyperbolic metrics such that $m_t$ converges to $m_\g\in\T(S_\g)$ in $\overline{\T(S)}$ as $t\to 0$. Let $\rho_t\colon\T^{q}_{\la}(S)\to\End(V)$ be a continuous family of irreducible representations classified by $m_t$ and weights $\mathbf{p}\in\{0,\ldots,N-1\}^s$ labelling the punctures $v_1$, \ldots, $v_s$ of $S$.

Then, as $t\to 0$, the representation \[\rho_t\circ\Theta^q_{\g,\la}:\T^{q}_{\la_\g}(S_\g)\rightarrow \End(V)\]
approaches
\[\bigoplus_{\ \mathbf{p_\g}}\rho^\mathbf{p_\g}_\g\colon\T^{q}_{\la_\g}(S_\g)\to\End(\bigoplus_{\ \mathbf{p_\g}}V_\mathbf{p_\g})\]
 where the direct sum is over all possible compatible weights $\mathbf{p_\g}$ on $S_\g$ and $\rho^\mathbf{p_\g}_\g\colon\T^{q}_{\la_\g}(S_\g)\to\End(V_\mathbf{p_\g})$ is an irreducible representation classified by the metric $m_\g$ and the weights $\mathbf{p_\g}$.
\end{thm}

\begin{proof}
We first suppose that $k=1$, that is, $\g$ consists of a single curve.

Let $P_1$, \ldots, $P_s$ be the central elements of $\T^q_\la(S)$ associated to the punctures $v_1$, \ldots, $v_s$ of $S$ and $P^\g_1$, \ldots, $P^\g_s$ be the central elements of $\T^q_{\la_\g}(S_\g)$ associated to the same punctures in $S_\g$. We also have the two central elements $P'$ and $P''$ associated to the two new punctures $v'$ and $v''$. Consider also the monomial $X_\mathbf{\g}\in\T^q_\la(S)$ associated to the retraction of $\g$ to a cycle in $G$. In practice, if $\g$ crosses $c_i$ times the edge $\la_i$ of $\la$ for $i=1,\ldots,n$, we have
\[
X_\g=[X^{c_1}_1\cdots X^{c_n}_n].
\] 
This element is the quantum analogue of the exponential graph length $x_\g$ discussed at the end of Section~\ref{extension}.
\begin{lem}
\label{punctures}
The elements defined above satisfy:
\begin{enumerate}
\item $\Theta^q_{\g,\la}(P^\g_i)=P_i$ for $i=1,\ldots,s$;
\item $\Theta^q_{\g,\la}(P')=\Theta^q_{\g,\la}(P'')=X_\g$.
\end{enumerate}
\end{lem}
\begin{proof} This is a consequence of the interpretation of $\Theta^q_{\g,\la}$ given in Remark~\ref{theta}.

For (1), let $\alpha_i$ be a small curve going around $v_i$ once in $S_\g\subset S$. If we identify $\alpha_i$ with its retraction to a cycle in $G'_\g$, we have that $P^\g_i=Y_{\alpha_i}$. Since $\pi(\alpha_i)$ corresponds to the retraction of $\alpha_i$ onto $G$, we see that 
\[\Theta^q_{\g,\la}(P^\g_i)=X_{\pi(\alpha_i)}=P_i.\]

For (2), let $\g'$ be a curve parallel to $\g$ such that $\g'$ is homotopic to $v'$ in $S_\g$. Then $\g'$ retracts to a cycle in $G'_\g$ and $P'=Y_\g$. In addition $\pi(\g')$ corresponds to the retraction of $\g'$, and hence of $\g$, onto $G$, so
\[\Theta^q_{\g,\la}(P')=X_{\pi(\g')}=X_\g.\]
The same argument holds for $P''$ if one considers a curve $\g''$ parallel to $\g$ and homotopic to $v''$ in $S_\g$.
\end{proof}

\begin{lem}
\label{skewcommutant}
There exists $\Delta\in\mathcal{T}^q_\lambda(S)$ such that $X_\g\Delta=q^4 \Delta X_\g$. 
\end{lem}
\begin{proof}
By Lemma~\ref{homological}, it suffices to find a path $\delta$ in $G$ such that $\widehat{\g}\cdot\widehat{\delta}=\pm 2$, where $\widehat{\g}$ and $\widehat{\delta}$ are (the classes in $H_1(\widehat{S})$ of) oriented $\eta$--anti-invariant lifts of $\g$ and $\delta$ to $\widehat{S}$. Then $\Delta=X^{\pm 1}_\delta$ will satisfy the lemma, since $X_\g X_\delta=q^{2\widehat{\g}\cdot\widehat{\delta}}X_\delta X_\g$.

If $\g$ is non separating, and since $\g$ is an essential simple closed curve, there exists $\delta$ another essential simple closed curve in $S$ which intersects $\g$ exactly once. In addition, one can choose representatives of $\g$ and $\delta$ not passing through the ramification points of the covering $p\colon\widehat{S}\to S$, that is, not passing through the vertices of $G$. Then, by construction, $\widehat{\delta}=p^{-1}(\delta)$ is such that $\widehat{\g}\cdot\widehat{\delta}=\pm 2|\g\cap\delta|=\pm 2$.

If $\g$ is separating then it divides the set of vertices of $G$ into two non-empty subsets. Let $\delta$ be an arc in $S$ with endpoints at vertices of $G$ on each side of $\g$ and intersecting $\g$ exactly once. Then $\widehat{\delta}=p^{-1}(\delta)$ is a closed curve in $\widehat{S}$ with a natural orientation and by construction $\widehat{\g}\cdot\widehat{\delta}=2$.
\end{proof}

Lemma~\ref{skewcommutant} implies that, under the action of $\rho_t(X_\g)$, $V$ decomposes into eigenspaces $V_i$ of dimension $N^{3g+p-4}$ with associated eigenvalues $c(t) q^{i}$, where  $i\in\left\{0,\ldots,N-1\right\}$ and $c(t)\in\C^*$. In addition \[\rho_t(X^N_\g)=\rho_t((X^N_1)^{c_1}\cdots(X^N_n)^{c_n})=x^{c_1}_1(t)\cdots x^{c_n}_n(t)\Id_V=x_\g(t)\Id_V\]
where $x_\g(t)$ is the exponential graph length of $\g$. Hence, after a shift by some $N$--th root of 1, we can consider that $c(t)=\sqrt[N]{x_\g(t)}\in\R_+$.

$P'$ and $P''$ are central in $\T^q_{\la_\g}(S_\g)$, so the eigenspaces of $\rho_t\circ\Theta^q_{\g,\la}(P')=\rho_t\circ\Theta^q_{\g,\la}(P'')=\rho_t(X_\g)$ are invariant under the action of $\rho_t\circ\Theta^q_{\g,\la}$. In other words
\[
\rho_t\circ\Theta^q_{\g,\la}=\bigoplus_{i}\rho^i_{t,\g},
\]
where $\rho^i_{t,\g}\colon\T^q_{\la_\g}(S_\g)\to\End(V_i)$ is such that
\[
\rho^i_{t,\g}(P')=\rho^i_{t,\g}(P'')=c(t)q^i\Id_{V_i}.
\]
For dimensional reasons these representations are irreducible by Theorem~\ref{BoLiu1}.

By Lemma~\ref{punctures}, we have that $\rho_t\circ\Theta^q_{\g,\la}(P^\g_j)=\rho_t(P_j)=q^{p_j}\Id_V$, hence
\[
\rho^i_{t,\g}(P^\g_j)=q^{p_j}\Id_{V_i},
\]
and by definition of $\Theta^q_{\g,\la}$,
\[
\rho_t\circ\Theta^q_{\g,\la}(Y^N_j)=\rho_t((X^N_1)^{k_1}\cdots(X^N_n)^{k_n})=x^{k_{j1}}_1(t)\cdots x^{k_{jn}}_n(t)\Id_{V},
\]
hence
\[
\rho^i_{t,\g}(Y^N_j)=x^{k_{j1}}_1(t)\cdots x^{k_{jn}}_n(t)\Id_{V_i}.
\]

Then, by Proposition~\ref{coordlimit}, as $t\rightarrow 0$, $x^{k_{j1}}_1(t)\cdots x^{k_{jn}}_n(t)\rightarrow y_j$, the shear parameters of $m_\g$, and by Lemma~\ref{limitlength}, $c_\g(t)\rightarrow 1$. This implies that, as $t\rightarrow 0$, $\rho^i_{t,\g}$ approaches the irreducible representation $\rho^i_\g$ of $\T^q_{\la_\g}(S_\g)$ which is classified by the weights $p_1$, \ldots, $p_s$ associated to the punctures $v_1$, \ldots, $v_s$, the weight $i$ associated to $v'$ and $v''$, and the hyperbolic metric $m_\g\in\T(S_\g)$.

Using Lemma~\ref{composition}, the case of a multicurve $\g=\{\g_1,\ldots,\g_k\}$ follows by induction on $k$.
\end{proof}



\section{Behavior under changes of coordinates: the quantum Teichm\"uller space}

\subsection{The quantum Teichm\"uller space}

We want to apply the results of the preceding section to the representations of the quantum Teichm\"uller space $\T^q(S)$. Let us first recall its construction as given in \cite{Liu1}. If $\la$ is an ideal triangulation of $S$, we denote by $\widehat{\T}^q_\la(S)$ the fraction division algebra of the Chekhov--Fock algebra $\T^q_\la(S)$. Chekhov and Fock constructed a family of isomorphisms $\Phi^q_{\la\la'}:\widehat{\T}^q_{\la'}(S)\rightarrow\widehat{\T}^q_\la(S)$, called \emph{(quantum) coordinate change isomorphisms}, defined for any two triangulations $\la$, $\la'$ of $S$. In particular, if $\la''$ is another triangulation, they satisfy the composition relation $\Phi^q_{\la\la''}=\Phi^q_{\la\la'}\circ\Phi^q_{\la'\la''}$. The main example is given by the case when $\la$ and $\la'$ differ by a diagonal exchange in an embedded square in $S$ as in Figure~\ref{diagexch}.
\begin{figure}[htb!]
\includegraphics{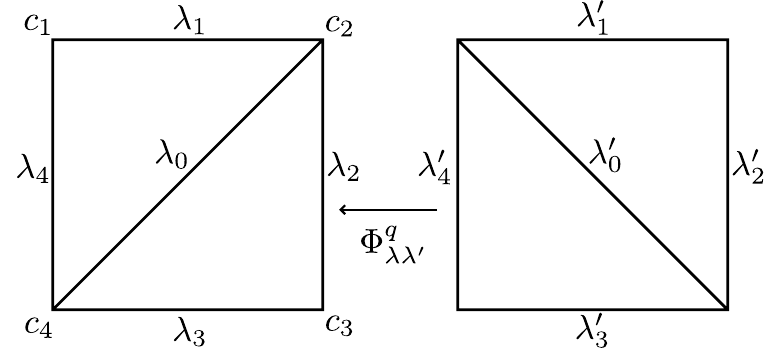}
\caption{\label{diagexch}}
\end{figure}
Then $\Phi^q_{\la\la'}(X'_n)=X_n$ if $n\neq1,\ldots,5$, $\Phi^q_{\la\la'}(X'_0)=X^{-1}_0$ and 
\begin{align*}
& \Phi^q_{\la\la'}(X'_1)=(1+qX_0)X_1, & \Phi^q_{\la\la'}(X'_2)=(1+qX^{-1}_0)^{-1}X_2,\\
& \Phi^q_{\la\la'}(X'_3)=(1+qX_0)X_3, & \Phi^q_{\la\la'}(X'_4)=(1+qX^{-1}_0)^{-1}X_4.
\end{align*}
We refer to \cite{Liu1} or \cite{BoLiu} for similar formulas when some of the edges of the square are identified. One can then construct $\Phi^q_{\la\la'}$ for any triangulations $\la$, $\la'$, using the composition relation and the fact that one can get from any triangulation $\la$ to another triangulation $\la'$ by a succession of diagonal exchanges and reindexings (see \cite{Penner1} for a proof). The fact that the maps so-obtained do not depend on the choice of a sequence of triangulations from $\la$ to $\la'$ is one of the achievements of \cite{CheFo}.

Using these isomorphisms we can construct the \emph{quantum Teichm\"uller space} $\T^q(S)$ as the quotient
\[\T^q(S)=\bigsqcup_\la\widehat{\T}^q_{\la}(S)/\sim\]
where the disjoint union is over all triangulations $\la$ of $S$, and the equivalence relation $\sim$ identifies $X'\in\widehat{\T}^q_{\la'}(S)$ to $\Phi^q_{\la\la'}(X')\in\widehat{\T}^q_{\la}(S)$.

\subsection{Representations of the quantum Teichm\"uller space}

A first attempt at defining a representation of $\T^q(S)$ would be to consider a family of representations $\rho_\la$ of $\widehat{\T}^q_\la(S)$ for every triangulation $\la$ of $S$ such that $\rho_\la\circ\Phi^q_{\la\la'}=\rho_{\la'}$ for every $\la$, $\la'$. However, one can easily check that such representations cannot be finite dimensional. On the other hand, when $q^2$ is an $N$--th root of unity, the Chekhov--Fock algebras $\T^q_\la(S)$ admit many finite dimensional representations. Hence, for our purpose, a representation of the quantum Teichm\"uller space $\T^q(S)$ will be a family of representations of $\T^q_\la(S)$, for every triangulation $\la$ of $S$, satisfying certain compatibility relations when changing triangulations.

More precisely, let $\rho_\la:\T^q_\la(S)\rightarrow\text{End}(V)$ be an algebra homomorphism satisfying the following condition: for every Laurent polynomial $X'\in\T^q_{\la'}(S)$, the rational fraction $\Phi^q_{\la\la'}(X')\in\widehat{\T}^q_{\la'}(S)$ can be written as 
\[\Phi^q_{\la\la'}(X')=P_1Q^{-1}_1=Q^{-1}_2P_2\]
where $P_1$, $Q_1$, $P_2$, $Q_2\in\T^q_\la(S)$ are Laurent polynomials for which $\rho_\la(Q_1)$ and $\rho_\la(Q_2)$ are invertible in $\End(V)$. If $\rho_\la$ satisfies such a condition, we say that the composition
\[\rho_\la\circ\Phi^q_{\la\la'}:\T^q_{\la'}(S)\rightarrow\text{End}(V)\]
\emph{makes sense} and is defined naturally as
\[\rho_\la\circ\Phi^q_{\la\la'}(X')=\rho_\la(P_1)\rho_\la(Q_1)^{-1}=\rho_\la(Q_2)^{-1}\rho_\la(P_2)\in\End(V).\]
One can check that this definition doesn't depend on the decomposition of $\Phi^q_{\la\la'}(X')$ as a quotient of polynomials, and that this indeed defines an algebra homomorphism.

\begin{defi}
A \emph{representation} $\rho=\left\{\rho_\la\right\}_\la$ of the quantum Teichm\"uller space $\T^q(S)$ over the vector space $V$ consists of the data of an algebra homomorphism $\rho_\la:\T^q_\la(S)\rightarrow\text{End}(V)$ for every triangulation $\la$ such that, for every $\la$, $\la'$, the representation $\rho_\la\circ\Phi^q_{\la\la'}:\T^q_{\la'}(S)\rightarrow\text{End}(V)$ makes sense and is equal to $\rho_{\la'}$.
\end{defi}

We will sometimes use the notation $\rho:\T^q(S)\rightarrow\text{End}(V)$ for such a representation, keeping in mind that $\rho$ consists in fact of a family of homomorphisms $\left\{\rho_\la:\T^q_\la(S)\rightarrow\text{End}(V)\right\}_\la$. Such representations were called representations of the polynomial core of $\T^q(S)$ in \cite{BoLiu}.

To prove that a family of representations $\left\{\rho_\la\right\}_\la$ of the Chekhov--Fock algebras is in fact a representation of $\T^q(S)$, one can use the following lemma (Lemma 25 in \cite{BoLiu}).

\begin{lem}
\label{criterion}
Let an algebra homomorphism $\rho_\la:\T^q_\la(S)\rightarrow\text{End}(V)$ be given for every ideal triangulation $\la$. Suppose that $\rho_\la\circ\Phi^q_{\la\la'}:\T^q_{\la'}(S)\rightarrow\text{End}(V)$ makes sense and is equal to $\rho_{\la'}$ whenever $\la$ and $\la'$ differ by a diagonal exchange or a re-indexing. Then $\rho=\left\{\rho_\la\right\}_\la$ is a representation of $\T^q(S)$.
\end{lem}

Given an irreducible representation $\rho_\la\colon\T^q_\la(S)\to\End(V)$ for some ideal triangulation $\la$, one can show that the composition $\rho_\la\circ\Phi^q_{\la\la'}\colon\T^q_\la(S)\to\End(V)$ makes sense for any other ideal triangulation $\la'$ and defines an irreducible representation $\rho_{\la'}$ of $\T^q_{\la'}(S)$. Such a family $\rho=\{\rho_\la\}_\la$ is thus called an \emph{irreducible representation} of $\T^q(S)$. 

In addition, if $\rho_\la$ is classified by the weights $\mathbf{p}=(p_1,\ldots,p_s)$ and the metric $m\in\T(S)$ expressed in the shear coordinates for $\la$,  then the representation $\rho_{\la'}$ is also classified by the same weights and the metric $m$ expressed in the shear coordinates for $\la'$. In this case we say that $\rho=\{\rho_\la\}_\la$ is the irreducible  representation of $\T^q(S)$ classified by $m$ and $\mathbf{p}$ (see Lemma 29 and Theorem 30 in \cite{BoLiu}).

Finally, if $m_t\in\T(S)$ is a continuous family of hyperbolic metrics and $\rho_{t,\la}$ is a continuous family of representations of $\T^q_\la(S)$ classified by $m_t$ and weights $\mathbf{p}$, we say that the representations $\rho_t=\{\rho_{t,\la}\}_\la$ obtained in this way form a \emph{continuous family of representations of} $\T^q(S)$.

\subsection{Main theorem}
In this section we restrict once again to the case when $q$ is an $N$--th root of unity with $N$ odd. The next step is to study how the decomposition obtained in Theorem~\ref{cheforep} is affected by changing the triangulation. Given two triangulations $\la$, $\la'$, we can consider the following diagram:
\begin{align}
\label{diagram}
\xymatrix{{\widehat{\T}^q_{\la'}(S)} \ar[r]^{\Phi^q_{\la\la'}} & {\widehat{\T}^q_{\la}(S)}\\
					{\widehat{\T}^q_{\la'_\g}(S_\g)}\ar[u]^{\Theta^q_{\g,\la'}} \ar[r]^{\Phi^q_{\la_\g\la'_\g}} & {\widehat{\T}^q_{\la\g}(S_\g)}\ar[u]_{\Theta^q_{\g,\la}}}
\end{align}
which is in general \emph{non-commutative}.

We focus on the case when $\la$ and $\la'$ differ only by a diagonal exchange in a square $Q$ as in Figure~\ref{diagexch}.

If the curve $\gamma$ never crosses $Q$ vertically or horizontally, that is, never crosses successively $\la_1$, $\la_0$, $\la_3$ or $\la_2$, $\la_0$, $\la_4$, the triangulations $\la_\g$ and $\la'_\g$ also differ by a diagonal exchange and can be identified outside of a square $Q_\g$ (cf Figure~\ref{nocrossing})

\begin{figure}[htb!]
\includegraphics{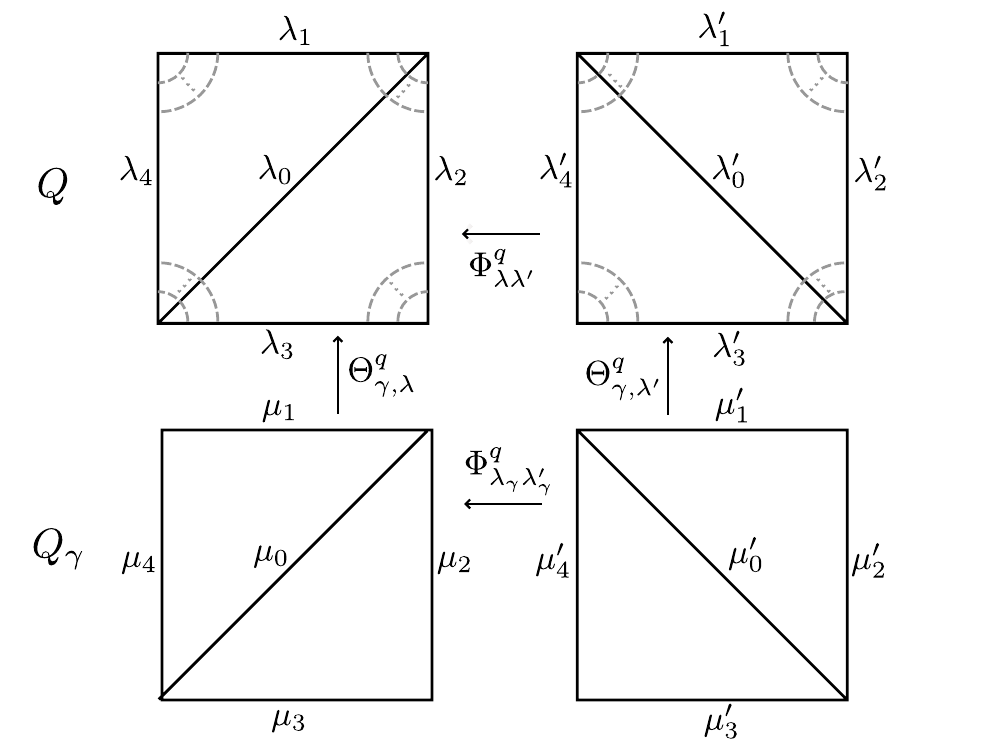}
\caption{\label{nocrossing}}
\end{figure}

\begin{figure}[htb!]
\includegraphics{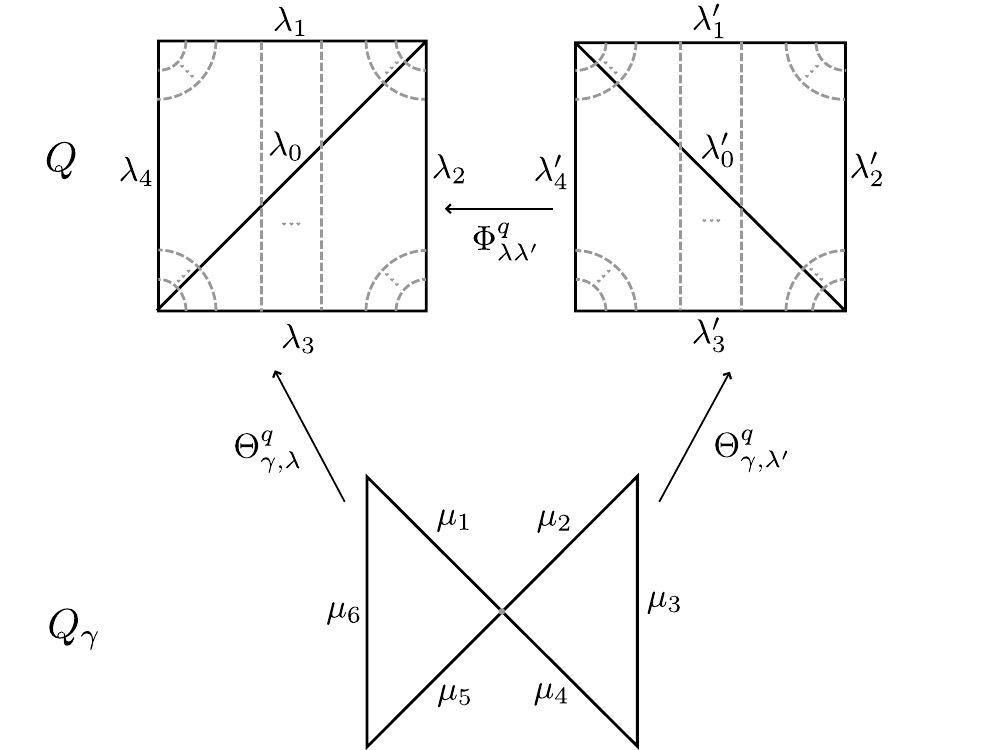}
\caption{\label{crossing}}
\end{figure}

If $\gamma$ does cross $Q$ vertically or horizontally, the triangulations $\la_\g$ and $\la'_\g$ can be identified (cf Figure~\ref{crossing}). In this case we introduce two maps $\Psi^q_{\la\la',v}$ and $\Psi^q_{\la\la',h}$, defined for each case of identifications of the boundary of $Q$ as follows: in each case $\Psi^q_{\la\la',v}(X'_i)=\Psi^q_{\la\la',h}(X'_i)=X_i$ for $i\geq 5$, $\Psi^q_{\la\la',v}(X'_0)=\Psi^q_{\la\la',h}(X'_0)=X^{-1}_0$ and
\begin{itemize}
\item if $Q$ is embedded,
\begin{align*}
& \Psi^q_{\la\la',v}(X'_1)=qX_0X_1, & \Psi^q_{\la\la',v}(X'_2)=X_2\\
& \Psi^q_{\la\la',v}(X'_3)=qX_0X_3, & \Psi^q_{\la\la',v}(X'_4)=X_4
\end{align*}
and
\begin{align*}
& \Psi^q_{\la\la',h}(X'_1)=X_1, & \Psi^q_{\la\la',h}(X'_2)=q^{-1}X_0X_2\\
& \Psi^q_{\la\la',h}(X'_3)=X_3, & \Psi^q_{\la\la',h}(X'_4)=q^{-1}X_0X_4;
\end{align*}
\item if $\la_1=\la_3$ and $\la_2\neq\la_4$,
\begin{align*}
& \Psi^q_{\la\la',v}(X'_1)=q^4X^2_0X_1, & \Psi^q_{\la\la',v}(X'_2)=X_2, \\
&                                       & \Psi^q_{\la\la',v}(X'_4)=X_4
\end{align*}
and
\begin{align*}
& \Psi^q_{\la\la',h}(X'_1)=X_1, & \Psi^q_{\la\la',h}(X'_2)=q^{-1}X_0X_2, \\
&                               & \Psi^q_{\la\la',h}(X'_4)=q^{-1}X_0X_4;
\end{align*}
\item if $\la_1=\la_2$ and $\la_3\neq\la_4$,
\begin{align*}
& \Psi^q_{\la\la',v}(X'_1)=X_0X_1, & \\
& \Psi^q_{\la\la',v}(X'_3)=qX_0X_3, & \Psi^q_{\la\la',v}(X'_4)=X_4
\end{align*}
and
\begin{align*}
& \Psi^q_{\la\la',h}(X'_1)=X_0X_1, & \\
& \Psi^q_{\la\la',h}(X'_3)=X_3, & \Psi^q_{\la\la',h}(X'_4)=q^{-1}X_0X_4;
\end{align*}
\item if $\la_1=\la_3$ and $\la_2=\la_4$, that is, $S$ is a once punctured torus,
\begin{align*}
& \Psi^q_{\la\la',v}(X'_1)=q^4X^2_0X_1,\\
& \Psi^q_{\la\la',v}(X'_2)=X_2
\end{align*}
and
\begin{align*}
& \Psi^q_{\la\la',h}(X'_1)=X_1,\\
& \Psi^q_{\la\la',h}(X'_2)=q^{-4}X^2_0X_2.
\end{align*}
\end{itemize}
The other cases are inverses of the ones above. Note that we exclude the case when $\la_1=\la_2$ and $\la_3=\la_4$ corresponding to a sphere with three holes since there are no essential simple closed curves on $S$ in this case.

One can easily check that $\Psi^q_{\la\la',v}$ and $\Psi^q_{\la\la',h}$ are algebra homomorphisms.

\begin{prop}
\label{inducedmap}
If $\la$ and $\la'$ differ by a diagonal exchange in a square $Q$ as in figure~\ref{diagexch} then one of the following is true:
\begin{enumerate}
\item the multicurve $\g$ doesn't cross $Q$ horizontally or vertically and
\[\Phi^q_{\la\la'}\circ\Theta^q_{\g,\la'}=\Theta^q_{\g,\la}\circ\Phi^q_{\la_\g\la'_\g};\]
\item $\g$ crosses $Q$ vertically at least once and
\[\Psi^q_{\la\la',v}\circ\Theta^q_{\g,\la'}=\Theta^q_{\g,\la};\]
\item $\g$ crosses $Q$ horizontally at least once and
\[\Psi^q_{\la\la',h}\circ\Theta^q_{\g,\la'}=\Theta^q_{\g,\la}.\]
\end{enumerate}
\end{prop}

\begin{proof}
We use the notations of Figures~\ref{diagexch}, \ref{nocrossing} and \ref{crossing}. The following lemma is a simple computation.
\begin{lem}
\label{monom}
Let $A_i$, $i=1,2,3,4$ monomials in $\T^q_\la(S)$ be the products of generators associated to the edges of $\la$ converging to the corner $c_i$ of $Q$, $A_5$ be the products of generators  when one crosses $Q$ vertically, $A_6$ the product of generators when one crosses $Q$ horizontally. Define similarly $B_1$, \ldots, $B_6$ monomials in $\T^q_{\la'}(S)$ for the triangulation $\la'$. Then, for $i,j=1,\ldots,4$
\[\Phi^q_{\la\la'}(\left[B_i\right])=\Psi^q_{\la\la',v}(\left[B_i\right])=\Psi^q_{\la\la',h}(\left[B_i\right])=\left[A_i\right]\]
and
\[\Psi^q_{\la\la',v}(\left[B_5\right])=\left[A_5\right]\text{ and } \Psi^q_{\la\la',h}(\left[B_6\right])=\left[A_6\right].\]
\end{lem}

Lemma~\ref{monom} says that $\Phi^q_{\la\la'}$, $\Psi^q_{\la\la',v}$ and $\Psi^q_{\la\la',h}$ send the products of generators at a corner of $Q$ for $\la'$ to the respective products for $\la$, \emph{respecting the quantum orderings}. In addition $\Psi^q_{\la\la',v}$ and $\Psi^q_{\la\la',h}$ respect the quantum ordered products of generators when crossing $Q$ vertically and horizontally respectively.

We first consider the case where both $Q$ and $Q_\g$ are embedded in $S$ and $S_\g$ respectively, and we look at the generators $Y_1$ and $Y'_1$ associated to the edges $\mu_1$ and $\mu'_1$. We recall that, by Remark~\ref{process}, the edges of $\la_\g$ can be identified with rectangles in the decomposition of $S_\g$ by $\la\smallsetminus\g$, formed by maximal chains of bigons.

For (1), we notice that $\mu_1$ corresponds to a rectangle for $\la\smallsetminus\g$ which starts along $\la_1$, and then crosses $l_i$ times $Q$ around the corner $c_i$ for $i=1,2,3,4$. It also crosses $k_i$ times $\la_i$ for $i=5,\ldots,n$. The same is true for $\mu'_1$. We let $Z=X^{k_5}_5\cdots X^{k_n}_n$ and $Z'=X^{'k_5}_5\cdots X^{'k_n}_n$. Then \[\Theta^q_{\g,\la}(Y_1)=\left[X_1A^{l_1}_1A^{l_2}_2A^{l_3}_3A^{l_4}_4Z\right]=q^\alpha X_1\left[A^{l_1}_1A^{l_2}_2A^{l_3}_3A^{l_4}_4Z\right]\] 
where $\alpha$ is some integer. We note that $\s(X_1,A_i)=\s(X'_1,B_i)$ for $i=1,2,3,4,$ hence
\[\Theta^q_{\g,\la}(Y'_1)=\left[X'_1B^{l_1}_1B^{l_2}_2B^{l_3}_3B^{l_4}_4Z'\right]=q^\alpha X'_1\left[B^{l_1}_1B^{l_2}_2B^{l_3}_3B^{l_4}_4Z'\right]\]
for the \emph{same} integer $\alpha$. We also note that $\mu_0$ corresponds to the edge $\la_0$ and hence $\Theta^q_{\g,\la}(Y_0)=X_0$. Using Lemma~\ref{monom} together with Lemma~\ref{quantumorder} we obtain
\begin{align*}
\Phi^q_{\la\la'}\circ\Theta^q_{\g,\la'}(Y'_1)&=\Phi^q_{\la\la'}(q^\alpha X'_1\left[B^{l_1}_1B^{l_2}_2B^{l_3}_3B^{l_4}_4Z'\right])\\ &=q^\alpha(1+qX_0)X_1\left[A^{l_1}_1A^{l_2}_2A^{l_3}_3A^{l_4}_4Z\right]\\
                                             &=\Theta^q_{\g,\la}((1+qY_0)Y_1)\\
                                             &=\Theta^q_{\g,\la}\circ\Phi^q_{\la_\g\la'_\g}(Y'_1).
\end{align*}
A similar computation works for $Y'_2$, $Y'_3$ and $Y'_4$. For $i>4$, $\mu_i$ corresponds to a rectangle with neither side ending along $Q$ but which may still cross it at the corners, and one shows in the same way that the equality holds for $Y'_5$, \ldots, $Y'_n$. Hence (1) is true in the case of embedded  squares.

For (2), $\mu_1$ corresponds to a rectangle for $\la\smallsetminus\g$ which starts along $\la_1$, crosses $l_i$ times the square $Q$ around the corner $c_i$ for $i=1,2,3,4$, and crosses $Q$ vertically $l_5$ times. For $\la'\smallsetminus\g$, it corresponds to a rectangle which starts along $\la'_0$, crosses $\la'_1$, then crosses $Q$ in the same way. It also crosses $k_i$ times $\la_i$ (resp. $\la'_i$), for $i=5,\ldots,n$. Then
\[\Theta^q_{\g,\la}(Y_1)=\left[X_1A^{l_1}_1A^{l_2}_2A^{l_3}_3A^{l_4}_4A^{l_5}_5Z\right]=q^\alpha X_1\left[A^{l_1}_1A^{l_2}_2A^{l_3}_3A^{l_4}_4A^{l_5}_5Z\right]\]
where $\alpha$ is some integer. We note that $\s(X_1,A_i)=\s(X'_0X'_1,B_i)$ for $i=1,2,3,4,5$, hence
\begin{align*}\Theta^q_{\g,\la}(Y'_1)=\left[X'_0X'_1B^{l_1}_1B^{l_2}_2B^{l_3}_3B^{l_4}_4B^{l_5}_5Z'\right]&=q^\alpha \left[X'_0X'_1\right]\left[B^{l_1}_1B^{l_2}_2B^{l_3}_3B^{l_4}_4B^{l_5}_5Z'\right]\\
                    &=q^{\alpha-1}X'_0X'_1\left[B^{l_1}_1B^{l_2}_2B^{l_3}_3B^{l_4}_4B^{l_5}_5Z'\right].
\end{align*}
Using the definition of $\Psi^q_{\la\la',v}$ in case 1, together with Lemma~\ref{monom} and Lemma \ref{quantumorder}, we obtain
\begin{align*}
\Psi^q_{\la\la',v}\circ\Theta^q_{\g,\la'}(Y'_1)&=\Psi^q_{\la\la',v}(q^{\alpha-1} X'_0X'_1\left[B^{l_1}_1B^{l_2}_2B^{l_3}_3B^{l_4}_4B^{l_5}_5Z'\right])\\ &=q^{\alpha-1}X^{-1}_0qX_0X_1\left[A^{l_1}_1A^{l_2}_2A^{l_3}_3A^{l_4}_4A^{l_5}_5Z\right]\\
&=q^\alpha X_1\left[A^{l_1}_1A^{l_2}_2A^{l_3}_3A^{l_4}_4A^{l_5}_5Z\right]\\
                                             &=\Theta^q_{\g,\la}(Y_1).
\end{align*}
A similar argument works for the other generators $Y'_2$, \ldots, $Y'_n$.

The case (3) is similar to (2), where $\mu_1$ corresponds to a rectangle which doesn't cross $Q$ vertically but crosses it horizontally $l_6$ times. The argument is the same replacing $B_5$, $l_5$ and $\Psi^q_{\la\la',v}$ with $B_6$, $l_6$ and $\Psi^q_{\la\la',h}$ respectively.

If some of the edges of $Q$ or $Q_\g$ are identified, the same method works using the corresponding formulae for $\Phi^q_{\la\la'}$, $\Psi^q_{\la\la',v}$ and $\Psi^q_{\la\la',h}$. One also needs to take into account that, in this case, an edge $\mu_i$ of $\la_\g$ may correspond to a rectangle in $S\smallsetminus\g$ with both ends along edges of $Q$. The formulae for $\Theta^q_{\g,\la}(Y_i)$ and $\Theta^q_{\g,\la}(Y'_i)$ have to be changed accordingly.
\end{proof}

The maps $\Psi^q_{\la\la',v}$ and $\Psi^q_{\la\la',h}$ can be interpreted as the ``limit'' of the coordinate change $\Phi^q_{\la\la'}$ when the length of $\g$ approaches 0, depending on whether $\g$ crosses $Q$ vertically or horizontally. This is made clearer by the following lemma.

\begin{lem}
\label{limitchangecoord}
Let $\rho_t$ be a continuous family of irreducible representations of $\T^q_\la(S)$ classified by a continuous family $m_t\in\T(S)$ such that, as $t\to 0$, $m_t$ approaches $m_\g\in\T(S_\g)$ in $\overline{\T(S)}$. Then, if $\la'$ differs from $\la$ by a diagonal exchange as in Figure~\ref{diagexch} and $\g$ crosses $Q$ vertically (i=v) or horizontally (i=h), we have
\[\rho_{t,\la}\circ\Phi^q_{\la\la'}\sim\rho_{t,\la}\circ\Psi^q_{\la\la',i}\ as\ t\rightarrow 0.\]
\end{lem}

\begin{proof}
We suppose that $(x_1(t),\ldots,x_n(t))$ are the shear parameters associated to $m_t$ for the triangulation $\la$. Then, as in section~\ref{convrep}, there are matrices $A_1$, \ldots, $A_n$ such that $\rho(X_i)=\sqrt[N]{x_i(t)}A_i$ for every $t$.

For simplicity, we consider the case when $Q$ is embedded in $S$. The computations are similar in the non-embedded cases. If $\g$ crosses $Q$ vertically then $x_0(t)\rightarrow\infty$ as $t\rightarrow 0$ and, considering for example the generator $X_1$, we have
\begin{align*}
\rho_{t,\la}\circ\Phi^q_{\la\la'}(X_1)&=(I+q\sqrt[N]{x_0(t)}A_0)\sqrt[N]{x_1(t)}A_1\\
                                      &\sim q\sqrt[N]{x_0(t)}\sqrt[N]{x_1(t)}A_0A_1=\rho_{t,\la}\circ\Psi^q_{\la\la',v}(X_1),
\end{align*}
and similarly for the other generators.

If $\g$ crosses $Q$ horizontally then $x_0(t)\rightarrow 0$ as $t\to 0$ and, considering for example the generator $X_2$, we have
\begin{align*}
\rho_{t,\la}\circ\Phi^q_{\la\la'}(X_2)&=(I+q(\sqrt[N]{x_0(t)}A_0)^{-1})^{-1}\sqrt[N]{x_2(t)}A_2\\
                                      &\sim q^{-1}\sqrt[N]{x_0(t)}\sqrt[N]{x_2(t)}A_0A_2=\rho_{t,\la}\circ\Psi^q_{\la\la',h}(X_2),
\end{align*}
and similarly for the other generators.
\end{proof}

\begin{prop}
\label{changetriang}
Suppose that $\rho_{t,\la}\colon\T^q_\la(S)\to\End(V)$ and $\rho_{t,\la'}\colon\T^q_{\la'}(S)\to\End(V)$ are two continuous families of irreducible representations classified by the same weights and by a continuous family $m_t\in\T(S)$ such that, as $t\rightarrow 0$, $m_t$ approaches $m_\g\in\T(S_\g)$ in $\overline{\T(S)}$. By Theorem~\ref{cheforep} and using the same notations, we have
\[
\lim_{t\to 0}\rho_{t,\la}\circ\Theta^q_{\g,\la}=\bigoplus_\mathbf{\ \,p_\g}\rho^\mathbf{p_\g}_{\g,\la}\ \ \text{and}\ \  \lim_{t\to 0}\rho_{t,\la'}\circ\Theta^q_{\g,\la'}=\bigoplus_\mathbf{\ \,p_\g}\rho^\mathbf{p_\g}_{\g,\la'}.
\]
Suppose in addition that, for all $t$, we have $\rho_{t,\la}\circ\Phi^q_{\la\la'}=\rho_{t,\la'}$. Then
\[\rho^\mathbf{p_\g}_{\g,\la'}=\rho^\mathbf{p_\g}_{\g,\la}\circ\Phi^q_{\la_\g\la'_\g}\ \text{for all compatible }\mathbf{p_\g}.
\]
\end{prop}

\begin{proof} We show the result for the case when $\g=\g_1$ is a simple closed curve. The general case follows by induction on $k$ using Lemma~\ref{composition}.

By Lemma~\ref{criterion}, it suffices to show the result for a diagonal exchange in a square $Q$ as in Figure~\ref{diagexch}. We let $P_\g$ be the quantum ordered product of generators associated to the edges of $\la_\g$ ending at one of the new punctures of $S_\g$. If $\la'_\g$ is different from $\la_\g$, we define similarly $P'_\g$. Following the proof of Theorem~\ref{cheforep}, we have that $\rho_{t,\la}\circ\Theta^q_{\g,\la}=\bigoplus_i\rho^i_{t,\la}$ where $\rho^i_{t,\la}$ are representations of $\T^q_{\la_\g}(S_\g)$ onto $V_i$, the eigenspace of $\rho_{t,\la}\circ\Theta^q_{\g,\la}(P_\g)$ with eigenvalue $q^i$. Then, by definition, $\rho^i_{\g,\la}$ is the limit as $t$ approaches 0 of $\rho^i_{t,\la}$.

We have 
\[\rho_{t,\la}\circ\Theta^q_{\g,\la}\circ\Phi^q_{\la_\g\la'_\g}\rightarrow(\bigoplus_i\rho^i_{\g,\la})\circ\Phi^q_{\la_\g\la'_\g}\]
and by hypothesis
\[\rho_{t,\la}\circ\Phi^q_{\la\la'}\circ\Theta^q_{\g,\la'}=\rho_{t,\la'}\circ\Theta^q_{\g,\la'}\rightarrow\bigoplus_i\rho^i_{\g,\la'}.\]

If $\g$ doesn't cross $Q$ vertically or horizontally, then $\Phi^q_{\la\la'}\circ\Theta^q_{\g,\la'}=\Theta^q_{\g,\la}\circ\Phi^q_{\la_\g\la'_\g}$ by Proposition~\ref{inducedmap}. Composing on both sides of this equation by $\rho_{t,\la}$ on the left and taking the limit as $t$ approaches 0, we get $\bigoplus_i\rho^i_{\g,\la'}=(\bigoplus_i\rho^i_{\g,\la})\circ\Phi^q_{\la_\g\la'_\g}$. Note that $\Phi^q_{\la_\g\la'_\g}(P'_\g)=P_\g$ so it sends the eigenspaces of $P'_\g$ onto those of $P_\g$ with the same eigenvalues. Hence $\rho^i_{\g,\la'}=\rho^i_{\g,\la}\circ\Phi^q_{\la_\g\la'_\g}$ for every $i$.

If $\g$ crosses $Q$, say vertically, then $\Phi^q_{\la_\g\la'_\g}$ is the identity. By Proposition~\ref{inducedmap}, $\Psi^q_{\la\la',v}\circ\Theta^q_{\g,\la'}=\Theta^q_{\g,\la}$ and, by Lemma~\ref{limitchangecoord}, $\rho_{t,\la}\circ\Phi^q_{\la\la'}\sim\rho_{t,\la}\circ\Psi^q_{\la\la',v}$. Composing on both sides of this equivalence by $\Theta^q_{\g,\la'}$ on the right and taking the limit as $t$ approaches 0 we get that $\bigoplus_i\rho^i_{\g,\la'}=\bigoplus_i\rho^i_{\g,\la}$. The decomposition $V=\bigoplus V_i$ is the same on each side, given by the eigenspaces of $P_\g=P_\g'$, hence $\rho^i_{\g,\la'}=\rho^i_{\g,\la}$ for all $i$ in this case.
\end{proof}

In the notations of Proposition~\ref{changetriang}, we see that $\rho^{\mathbf{p}_\g}_{\g,\la}$ differs from $\rho^{\mathbf{p}_\g}_{\g,\la'}$ only if $\la_\g$ is different from $\la'_\g$. Hence we can rename these representations $\rho^{\mathbf{p}_\g}_{\g,\la_\g}$.

If $\Omega$ is a subset of the set $\Lambda(S)$ of ideal triangulations of $S$ and if $\left\{\rho_\mu\right\}_{\mu\in\Omega}$ is a family of compatible representations of the Chekhov-Fock algebras $\T^q_\mu(S)$, we can \emph{extend} this family to a represention $\rho=\left\{\rho_\la\right\}_{\la\in\Lambda(S)}$ of $\T^q(S)$ by setting $\rho_\la=\rho_\mu\circ\Phi^q_{\mu\la}$ for any $\la\in\Lambda(S)$, for some fixed $\mu\in\Omega$. The composition rule for the coordinate change isomorphisms implies that this definition doesn't depend on $\mu$. Of particular interest here is the subset $\Lambda_\g=\left\{\lambda_\g\text{ induced by }\la\ |\ \la\in\Lambda(S)\right\}$ of the set $\Lambda(S_\g)$ of ideal triangulations of $S_\g$. We believe at this point the two sets coincide, but assume for now that they may be distinct.

By Proposition~\ref{changetriang}, if $\rho_t=\left\{\rho_{t,\la}\right\}_\la$ is such a continuous family of representations of $\T^q(S)$, the limiting irreducible factors $\rho^\mathbf{p_\g}_{\g,\la_\g}$ for each $\rho_{t,\la}$, as given by Theorem~\ref{cheforep}, satisfy the compatibility relations  $\rho^\mathbf{p_\g}_{\g,\la'_\g}=\rho^\mathbf{p_\g}_{\g,\la_\g}\circ\Phi^q_{\la_\g\la'_\g}$. Hence, taken together, they can be extended to form irreducible representations of $\T^q(S_\g)$, proving the following theorem.

\begin{thm}
Let $\rho_t=\left\{\rho_{t,\la}\right\}_\la$ be a continuous family of irreducible representations of $\T^q(S)$ classified by weights $\mathbf{p}\in\left\{0,\ldots,N-1\right\}^s$ and a continuous family of metrics $m_t\in\T(S)$ such that $m_t$ approaches $m_\g\in\T(S_\g)$ in $\overline{\T(S)}$ as $t\rightarrow 0$. For each triangulation $\la$, we let the limit
\[
\lim_{t\to 0}\rho_{t,\la}\circ\Theta^q_{\g,\la}=\bigoplus_\mathbf{\ p_\g}\rho^{\mathbf{p_\g}}_{\g,\la_\g}
\]
 be given as in Theorem~\ref{cheforep}.

Then, for every compatible weights $\mathbf{p_\g}$ on $S_\g$, the family of representations $\{\rho^\mathbf{p_\g}_{\g,\la_\g}\}_{\la_\g\in\Lambda_\g}$ extends to an irreducible representation  $\rho^\mathbf{p_\g}_\g$ of $\T^q(S_\g)$ classified by the metric $m_\g$ and the weights $\mathbf{p_\g}$.
\end{thm}


\end{document}